\documentclass{article}
\usepackage{a4}
\usepackage{amsfonts,amsthm,amsmath}
\usepackage{mathrsfs}
\usepackage{lunadiagrams}
\newtheorem{theorem}{Theorem}[subsection]
\newtheorem{proposition}[theorem]{Proposition}
\newtheorem{lemma}[theorem]{Lemma}
\newtheorem{corollary}[theorem]{Corollary}
\theoremstyle{definition}
\newtheorem{definition}[theorem]{Definition}
\newtheorem{remark}[theorem]{Remark}
\newtheorem{example}[theorem]{Example}

\numberwithin{equation}{subsection}

\renewcommand{\S}{\mathscr S}
\newcommand{\A}{\mathnormal A}
\newcommand{\sr}{\mathnormal\Sigma}
\newcommand{\co}{\mathnormal\Delta}

\newcommand{\comb}{\mathrm{-comb}}

\newcommand{\fR}{\textsf{R}}
\newcommand{\fS}{\textsf{S}}
\newcommand{\fT}{\textsf{T}}
\newcommand{\ppc}{ppc}
\newcommand{\faa}{\textsf{aa}}
\newcommand{\fo}{\textsf{2a}}

\newcommand{\fm}{\textsf{m}}
\newcommand{\fc}{\textsf{c}}
\newcommand{\fp}{\textsf{p}}
\newcommand{\fq}{\textsf{q}}
\newcommand{\fx}{\textsf x}
\newcommand{\fy}{\textsf y}
\newcommand{\fz}{\textsf z}
\newcommand{\fu}{\textsf u}
\newcommand{\fv}{\textsf v}
\newcommand{\fw}{\textsf w}
\newcommand{\fma}{\textsf{a}}
\newcommand{\fmb}{\textsf{b}}
\newcommand{\bd}{\mathsf{bd}}
\newcommand{\bc}{\mathsf{bc}}
\newcommand{\ab}{\mathsf{ab}}
\renewcommand{\a}{\mathsf{a}}
\renewcommand{\b}{\mathsf{b}}
\renewcommand{\c}{\mathsf{c}}
\renewcommand{\d}{\mathsf{d}}
\newcommand{\f}{\mathsf{f}}
\newcommand{\g}{\mathsf{g}}
\newcommand{\ad}{\mathsf{ad}}
\newcommand{\ag}{\mathsf{ag}}
\newcommand{\bg}{\mathsf{bg}}
\newcommand{\dg}{\mathsf{dg}}
\newcommand{\ac}{\mathsf{ac}}
\newcommand{\e}{\mathsf{e}}

\newcommand{\pssclanlabel}{}
\newcounter{pssnumber}
\newcommand{\pssclan}[1]{\subsubsection*{Clan $\mathbf{#1}$}\renewcommand{\pssclanlabel}{#1}\setcounter{pssnumber}{0}}

\newcommand{\pssfamily}[1]{\bigskip\noindent {\bf #1} \smallskip}
\newcommand{\pss}{\refstepcounter{pssnumber}\noindent{\bf $\mathbf\pssclanlabel$--$\mathbf\thepssnumber$.}\ }
\newcommand{\pssd}{\noindent{\bf $\mathbf\pssclanlabel$--$\mathbf{\thepssnumber'}$.}\ }

\begin{document}

\title{Primitive spherical systems}
\author{P.\ Bravi} 
\date{}
\maketitle

\begin{abstract}
A spherical system is a combinatorial object, arising in the theory of wonderful varieties, defined in terms of a root system. All spherical systems can be obtained by means of some general combinatorial procedures (such as parabolic induction and wonderful fiber product) from the so-called primitive spherical systems. Here we classify primitive spherical systems. As an application, we prove that the quotients of a spherical system are in correspondence with the so-called distinguished subsets of colors.
\end{abstract}

\section*{Introduction}

Let $G$ be a semisimple group over the complex numbers. The spherical $G$-systems are combinatorial objects defined axiomatically in terms of the root system of $G$.

Let $X$ be a wonderful $G$-variety: a smooth projective quasi-homogeneous $G$-variety such that the irreducible components of the $G$-boundary are quasi-homogeneous of codimension 1 and have non-empty transversal intersection, see \cite{Lu01}. One can associate  to $X$ a special combinatorial invariant which satisfies the axioms of spherical $G$-system. The Luna conjecture states that wonderful $G$-varieties are classified by spherical $G$-systems. 

Without loss of generality one can assume that the group $G$ is adjoint.

Many partial results related to the conjecture are now known, see \cite{Lu01, P03, BP05, Br07, BCF08, Lo09, BCF09}. See also \cite{CF08,CF09}.

D.~Luna's original proof of the conjecture in type $\mathsf A$ is based on some reduction steps to some primitive spherical systems and the case-by-case proof that all primitive spherical systems correspond to one and only one wonderful variety. 

This approach works for any group $G$ (see \cite{BP09}) and therefore a way to prove the conjecture consists in checking the correspondence on the primitive spherical systems.

Here we explicitly classify primitive spherical systems for any group $G$ of adjoint type. 

This classification allows to describe the structure of all spherical systems, and it may reveal itself useful not only in the proof of the Luna conjecture.

For instance, we also analyze the quotients of a general spherical system. Indeed, we determine all the quotients of primitive spherical systems and prove that in general all the so-called distinguished subsets of colors are good, that is, give rise to well-defined quotient spherical systems. This was conjectured in \cite{Lu01} as well.

\bigskip
Without referring to the group $G$, spherical $G$-systems will just be called spherical $R$-systems of adjoint type, where $R$ is the root system of $G$.

The list presented here includes some already known classes of primitive spherical $R$-systems: those with $R$ of type $\mathsf {A\ D\ E}$ (\cite{Lu01,BP05,Br07}), of type $\mathsf {A\ C}$ (\cite{P03}), of type $\mathsf F$ (\cite{BL09}) or the so-called strict primitive spherical systems (\cite{BCF09}). 

A remark is necessary: our present definition of primitive spherical system is slightly more restrictive than that of the previous papers (\cite{BP05,Br07}), see Definitions \ref{def:dec} and \ref{def:pss}. 

\bigskip
In Section~\ref{sec:def} we give the definition of spherical system and of other basic notions; for a more accessible introduction see \cite{Lu01,BL09}. 

In Section~\ref{sec:cla} we give the definition of primitive spherical system and of spherical system with a primitive positive 1-comb. We give the list of primitive spherical systems and of spherical systems with a primitive positive 1-comb. In \ref{subsec:col} we prove that such lists are complete.

In Section~\ref{sec:prop} we use the classification obtained in the preceding section to prove that all the distinguished subsets of colors are good.

In Appendix~\ref{app:low} we give all spherical systems of rank $\leq2$, with their minimal quotients.

In Appendix~\ref{app:qpss} we give all the minimal quotients of primitive spherical systems and of spherical systems with a primitive positive 1-comb (with rank $>2$).

\bigskip
The author is indebted to D.~Luna and G.~Pezzini for their help and support.

\section{Basic Definitions}\label{sec:def}

\subsection{Spherical systems}

Let $R$ be a reduced root system. Let $S$ be a set of simple roots of $R$. Simple roots of irreducible root systems will be labelled as in \cite{Bo}.

Let $\sigma$ be an element of $\mathbb N S$, say $\sigma=\sum_{\alpha\in S}n_\alpha\alpha$, define $\mathrm{supp}\,\sigma$ to be the subset of simple roots $\alpha$ such that $n_\alpha\neq0$. 

\begin{definition}\label{def:sr}
The set of spherical $R$-roots of adjoint type, denoted by $\Sigma_{\mathrm{ad}}(R)$, is the set of $\sigma\in\mathbb N S$ such that 
\begin{itemize}
\item $\sigma=\alpha+\beta$ where $\alpha$ and $\beta$ are orthogonal simple roots ($\sigma$ is said to be of type $aa$), 
\item or $\mathrm{supp}\,\sigma$ is the set of simple roots of an irreducible root subsystem and, after restricting $S$ to $\mathrm{supp}\,\sigma$, $\sigma$ is one of the following:
\begin{center}
\begin{tabular}{lll}
type of supp & $\sigma$ & type of $\sigma$\\
\hline
$\mathsf A_n$, $n\geq 1$ & $\sum_{i=1}^n\alpha_i$ & $a(n)$\\
$\mathsf A_1$ & $2\alpha_1$ & $2a$\\
$\mathsf B_n$, $n\geq 2$ & $\sum_{i=1}^n\alpha_i$ & $b(n)$\\
 & $\sum_{i=1}^n2\alpha_i$ & $2b(n)$\\
$\mathsf B_3$ & $\alpha_1+2\alpha_2+3\alpha_3$\\
$\mathsf C_n$, $n\geq 3$ & $\alpha_1+(\sum_{i=2}^{n-1}2\alpha_i)+\alpha_n$ & $c(n)$\\
$\mathsf D_n$, $n\geq 3$ & $(\sum_{i=1}^{n-2}2\alpha_i)+\alpha_{n-1}+\alpha_n$ & $d(n)$\\
$\mathsf F_4$ & $\alpha_1+2\alpha_2+3\alpha_3+2\alpha_4$ & $f$\\
$\mathsf G_2$ & $2\alpha_1+\alpha_2$ & $g$\\
 & $4\alpha_1+2\alpha_2$ & $2g$\\
 & $\alpha_1+\alpha_2$ \\
\end{tabular}
\end{center}
\end{itemize}
\end{definition}

We also set $d(2)=aa$, $b(1)=a(1)$, $2b(1)=2a$, $c(2)=b(2)$.

\begin{definition}
Let $S^p$ be a subset of $S$ and $\sigma\in\Sigma_{\mathrm{ad}}(R)$ a spherical $R$-root of adjoint type, $S^p$ and $\sigma$ are said to be compatible if $S^{pp}(\sigma)\subset S^p\subset S^p(\sigma)$ where $S^p(\sigma)$ is the set of simple roots orthogonal to $\sigma$ and $S^{pp}(\sigma)$ is equal to
\begin{itemize}
\item $S^p(\sigma)\cap\mathrm{supp}\,\sigma\setminus\{\alpha_n\}$, if $\sigma=\sum_{i=1}^n\alpha_i$ with support of type $\mathsf B_n$,
\item $S^p(\sigma)\cap\mathrm{supp}\,\sigma\setminus\{\alpha_1\}$, if $\sigma$ has support of type $\mathsf C_n$, 
\item $S^p(\sigma)\cap\mathrm{supp}\,\sigma$, otherwise.
\end{itemize}
\end{definition}

\begin{definition}
A triple $\S=(S^p,\sr,\A)$, where
\begin{itemize}
\item $S^p\subset S$
\item $\sr\subset \Sigma_{\mathrm{ad}}(R)$ without proportional elements
\item $\A$ is a finite set endowed with a pairing $c\colon\A \times\sr\to\mathbb Z$ (for all $\alpha\in S\cap\sr$ set $\A(\alpha)=\{D\in\A: c(D,\alpha)=1\}$),
\end{itemize}
is called a spherical $R$-system of adjoint type if
\begin{itemize}
\item[(A1)] for all $D \in \A$ and $\sigma \in \sr$, $c(D,\sigma)\leq 1$, and if $c(D,\sigma)=1$ then $\sigma\in S$; 
\item[(A2)] for all $\alpha \in S\cap\sr$, $\mathrm{card}(\A(\alpha))=2$, and if $\A(\alpha)=\{D_\alpha^+,D_\alpha^-\}$ then $c(D_\alpha^+,\sigma)+ c(D_\alpha^-,\sigma)= \langle \alpha^\vee , \sigma \rangle$ for all $\sigma\in\sr$;
\item[(A3)] $\A=\cup_{\alpha\in S\cap\sr}\A(\alpha)$;
\item[($\Sigma 1$)] if $2\alpha \in 2S \cap\sr$ then $\frac{1}{2}\langle\alpha^\vee, \sigma\rangle\in\mathbb Z_{\leq0}$ for all $\sigma \in \sr \setminus \{ 2\alpha \}$; 
\item[($\Sigma 2$)] if $\alpha$ and $\beta$ are orthogonal simple roots with $\alpha + \beta \in \sr$ then $\langle \alpha ^\vee , \sigma \rangle = \langle \beta ^\vee , \sigma \rangle$ for all $\sigma \in \sr$; 
\item[(S)] $S^p$ and $\sigma$ are compatible for all $\sigma \in \sr$. 
\end{itemize}
\end{definition}

The elements of $\sr$ are called spherical roots of $\S$. The cardinality of $\sr$ is called the rank of $\S$. The pairing $c\colon\A\times\sr\to\mathbb Z$ is called restricted Cartan pairing.

In the following with spherical $R$-system we will mean spherical $R$-system of adjoint type.

\begin{definition}
Let $\S=(S^p,\sr,\A)$ be a spherical $R$-system. Set $S^p\sqcup S^a\sqcup S^{2a}\sqcup S^b=S$ such that $S^a=S\cap\sr$ and $S^{2a}=2S\cap\sr$.
The set of colors $\co$ of $\S$ is the finite set $\co=\co^a\sqcup\co^{2a}\sqcup\co^b$ endowed with a pairing $c\colon \co\times\sr\to\mathbb Z$ as follows
\begin{itemize}
\item $\co^a=\A$, 
\item there exists a bijective map $2\alpha\mapsto D_{2\alpha}$ from $S^{2a}$ to $\co^{2a}$, with $c(D_{2\alpha},\sigma)=\frac{1}{2}\langle\alpha^\vee,\sigma\rangle$ for all $\sigma\in\sr$,
\item there exists a surjective map $\alpha\mapsto D_{\alpha}$ from $S^b$ to $\co^b$ such that $D_\alpha=D_\beta$ if and only if $\alpha\perp\beta$ and $\alpha+\beta\in\sr$, with $c(D_\alpha,\sigma)=\langle\alpha^\vee,\sigma\rangle$ for all $\sigma\in\sr$.
\end{itemize} 
Furthermore, set $\co(\alpha)=\A(\alpha)=\{D^+_\alpha,D^-_\alpha\}$ if $\alpha\in S^a$, $\co(\alpha)=\{D_{2\alpha}\}$ if $\alpha\in S^{2a}$ and $\co(\alpha)=\{D_\alpha\}$ if $\alpha\in S^b$.
\end{definition}

\begin{definition}
The difference between the cardinality of $\co$ and the cardinality of $\sr$ (the rank of $\S$) is called defect and denoted by $\mathrm d(\S)$.
\end{definition}

\begin{remark}
From the classification of rank 2 spherical systems (see Appendix~\ref{app:low} or \cite{W96}) it turns out that all pairs of spherical roots of any spherical system have non-positive scalar product. This implies that the spherical roots of any spherical system are linearly independent\footnote{We thank A.~Maffei for having pointed out to us this easy argument.}. This implies also that $\sr$ is a basis of a reduced root system, but the latter plays no role in the present paper.
\end{remark}

The above pairing $c$, called Cartan pairing, can be extended by $\mathbb Z$-linearity to a pairing $c\colon\mathbb Z\co\times\mathbb Z\sr\to\mathbb Z$.

\subsection{Localization and induction}

\begin{definition}
Let $\S=(S^p,\sr,\A)$ be a spherical $R$-system. Let $S'$ be a subset of simple roots and $R'$ the corresponding root subsystem. The spherical $R'$-system obtained from $\S$ by localization is $\S'=((S')^p,\sr',\A')$:
\begin{itemize}
\item $(S')^p=S^p\cap S'$,
\item $\sr'=\{\sigma\in\sr:\mathrm{supp}\,\sigma\subset S'\}$,
\item $\A'=\cup\A(\alpha)$ for all $\alpha\in \sr\cap S'$.
\end{itemize} 
\end{definition}

\begin{definition}
A spherical $R$-system $\S=(S^p,\sr,\A)$ is called cuspidal if $\mathrm{supp}(\sr)=S$. 
\end{definition}

\begin{definition}\label{def:ind}
Let $\S'=((S')^p,\sr',\A')$ be a spherical $R'$-system and let $S'\subset S$. The spherical $R$-system obtained from $\S'$ by induction is $\S=((S')^p,\sr',\A')$.
\end{definition}

\subsection{Quotients}\label{subsec:quo}

Let $\S=(S^p,\sr,\A)$ be a spherical $R$-system with set of colors $\co$. An element $D$ of $\mathbb Z\co$ is called positive if $c(D,\sigma)\geq0$ for all $\sigma\in\sr$.

\begin{definition}
A subset $\co'\subset\co$ of colors of $\S$ is called distinguished if there exists a positive element in $\mathbb N_{>0}\co'$.
\end{definition}

Consider the set $\sr/\co'$ of minimal generators of the semigroup 
\[\{\sigma\in\mathbb N\sr : c(D,\sigma)=0\ \forall D\in\co'\}\]
and the triple $\S/\co'=(S^p/\co',\sr/\co',\A/\co')$ where $S^p/\co'=\{\alpha\in S : \co(\alpha)\subset\co'\}$ and $\A/\co'$ is the subset of $\A$ given by $S\cap\sr/\co'$.

\begin{definition}
A distinguished subset $\co'\subset\co$ is called good if $\S/\co'$ is a spherical $R$-system. In this case $\S/\co'$ is called quotient spherical $R$-system.
\end{definition}

The set of colors of the quotient spherical system $\S/\co'$ can be identified with $\co\setminus\co'$. 

\begin{definition}
A good distinguished subset $\co'$ is called homogeneous if $\sr/\co'=\emptyset$.
\end{definition}

A quotient of a spherical system given by a minimal good distinguished subset of colors will be called minimal. 

\subsection{Luna diagrams}

Let us recall how to visualize a spherical system via its Luna diagram. 

Spherical roots are represented on the Dynkin diagram of their support as in Table~\ref{tab:dsr}.

\begin{table}\caption{Spherical roots}\label{tab:dsr}
\begin{center}
\begin{tabular}{|cl|}
\hline
diagram&spherical root\\
\hline
\begin{picture}(3000,1500)\put(300,-600){\line(1,0){2400}}\multiput(0,0)(2400,0){2}{\put(300,0){\line(0,-1){600}}\put(300,300){\circle{600}}\put(300,300){\circle*{150}}}\end{picture}&$\alpha_1+\alpha_1'$\\
\begin{picture}(600,2100)\put(300,300){\usebox{\aone}}\end{picture}& $\alpha_1$\\
\begin{picture}(6000,2100)\put(300,300){\usebox{\mediumam}}\end{picture}&$\sum_{i=1}^n\alpha_i$, $n\geq 2$\\
\begin{picture}(600,2100)\put(300,300){\usebox{\aprime}}\end{picture}&$2\alpha_1$\\
\begin{picture}(7800,2100)\put(300,300){\usebox{\shortbm}}\end{picture}&$\sum_{i=1}^n\alpha_i$, $n\geq 2$\\
\begin{picture}(7800,2100)\put(300,300){\usebox{\shortbprimem}}\end{picture}&$\sum_{i=1}^n2\alpha_i$, $n\geq 2$\\
\begin{picture}(4200,2100)\put(300,300){\usebox{\bthirdthree}}\end{picture}&$\alpha_1+2\alpha_2+3\alpha_3$\\
\begin{picture}(9000,2100)\put(0,300){\usebox{\shortcm}}\end{picture}&$\alpha_1+(\sum_{i=2}^{n-1}2\alpha_i)+\alpha_n$, $n\geq 3$\\
\begin{picture}(6900,2400)\put(300,300){\usebox{\shortdm}}\end{picture}&
$(\sum_{i=1}^{n-2}2\alpha_i)+\alpha_{n-1}+\alpha_n$, $n\geq 3$\\
\begin{picture}(6000,2400)\put(300,300){\usebox{\ffour}}\end{picture}&$\alpha_1+2\alpha_2+3\alpha_3+2\alpha_4$\\
\begin{picture}(2400,2100)\put(300,300){\usebox{\gtwo}}\end{picture}&$2\alpha_1+\alpha_2$\\
\begin{picture}(2400,2100)\put(300,300){\usebox{\gprimetwo}}\end{picture}&$4\alpha_1+2\alpha_2$\\
\begin{picture}(2400,2100)\put(300,300){\usebox{\gsecondtwo}}\end{picture}&$\alpha_1+\alpha_2$\\
\hline
\end{tabular}
\end{center}
\end{table}

Further not shadowed circles around vertices may occur: $S^p$ equals the set of vertices having no circles around, below or above. 

The set $S\cap\sr$ corresponds to the set of vertices which have circles above and below. For each $\alpha\in S\cap\sr$, these two circles are identified with the elements of $\A(\alpha)$, the circle above to $D_\alpha^+$, where $D_\alpha^+$ is such that $c(D_\alpha^+,\sigma)\in\{1,0-1\}$, for every spherical root $\sigma$. Circles in different $\A(\alpha)$'s are joined by a line, if they correspond to the same element in $\A$. Finally, for every spherical root $\sigma$ not orthogonal to $\alpha$ such that $c(D_\alpha^+, \sigma) = -1$, there is an arrow (usually only a symbol ``$<$'' or ``$>$''), starting from the circle corresponding to $D_\alpha^+$ and pointing toward $\sigma$. The set $\A$ and the restricted Cartan pairing $c\colon \A \times\sr\to \mathbb Z$ can then be recovered by Axiom~A2.

An arrow from the diagram of $\S$ to the diagram of $\S'$ means that the spherical system $\S$ has $\S'$ as quotient spherical system. Minimal quotients with decreasing defect are denoted by dashed arrows, minimal quotients with non-decreasing defect are denoted by continuous arrows.

\subsection{Spherical systems of rank $\leq2$}\label{subsec:low}

The spherical systems of rank 1 and of rank 2 are known (\cite{W96}), their classification can directly be deduced from the axioms of spherical systems.

In Appendix~\ref{app:low} we give the Luna diagrams of all cuspidal spherical systems of rank 1 and of all cuspidal spherical systems of rank 2 that are not direct product of two spherical systems of rank 1 (see the appendix for a precise definition of direct product). We provide also their minimal quotient spherical systems.

\begin{remark}\label{rem:lgood}
All distinguished subsets of colors of spherical systems of rank $\leq2$ are good. 
\end{remark}

\section{Classification of spherical systems}\label{sec:cla}

\subsection{Primitive spherical systems}\label{subsec:pss}

\subsubsection*{Decomposition}

\begin{definition}\label{def:dec}
Let $\S=(S^p,\sr,\A)$ be a spherical $R$-system. Let $\co'$ and $\co''$ be good distinguished subsets such that
\begin{itemize}
\item $(S^p/\co'\setminus S^p)\perp (S^p/\co''\setminus S^p)$ and
\item $\sr\subset(\sr/\co'\cup\sr/\co'')$. 
\end{itemize}
Then we say that $\co'$ and $\co''$ decompose\footnote{Notice that this definition is less restrictive than that of the previous papers (\cite{BP05,Br07}).} $\S$, and that $\S$ is fiber product of $\S/\co'$ and $\S/\co''$ over $\S/(\co'\cup\co'')$.
\end{definition} 

\begin{example}
\[\begin{picture}(21000,9750)
\put(8700,8400){
\multiput(0,0)(1800,0){2}{\usebox{\edge}}
\multiput(0,0)(1800,0){3}{\usebox{\aone}}
\put(0,600){\usebox{\toe}}
\put(3600,600){\usebox{\tow}}
\multiput(0,1350)(3600,0){2}{\line(0,-1){450}}
\put(0,1350){\line(1,0){3600}}
}
\put(7800,7500){\vector(-2,-1){3000}}
\put(300,4650){
\multiput(0,0)(1800,0){2}{\usebox{\edge}}
\multiput(0,0)(3600,0){2}{\usebox{\aone}}
\put(1800,0){\usebox{\wcircle}}
\multiput(0,1350)(3600,0){2}{\line(0,-1){450}}
\put(0,1350){\line(1,0){3600}}
}
\put(13200,7500){\multiput(0,0)(600,-300){5}{\multiput(0,0)(20,-10){15}{\line(1,0){20}}}\put(2700,-1350){\vector(2,-1){300}}}
\put(17100,4650){
\multiput(0,0)(1800,0){2}{\usebox{\edge}}
\multiput(0,0)(3600,0){2}{\usebox{\wcircle}}
\put(1800,0){\usebox{\aone}}
\multiput(0,-1350)(3600,0){2}{\line(0,1){1050}}
\put(0,-1350){\line(1,0){3600}}
}
\put(4800,3300){\multiput(0,0)(600,-300){5}{\multiput(0,0)(20,-10){15}{\line(1,0){20}}}\put(2700,-1350){\vector(2,-1){300}}}
\put(16200,3300){\vector(-2,-1){3000}}
\put(8700,900){
\multiput(0,0)(1800,0){2}{\usebox{\edge}}
\multiput(0,0)(1800,0){3}{\usebox{\wcircle}}
\multiput(0,-900)(3600,0){2}{\line(0,1){600}}
\put(0,-900){\line(1,0){3600}}
}
\end{picture}\]
\end{example}

\subsubsection*{Positive combs}

\begin{definition}
Let $\S=(S^p,\sr,\A)$ be a spherical $R$-system. An element $D$ of $\A$ that is positive in the sense of \ref{subsec:quo} is called positive $n$-comb, where $n$ equals the cardinality of $S_D=\{\alpha\in S : D\in\co(\alpha)\}$. 
\end{definition}

By abuse of terminology a spherical system with $\sr=S$ and a positive comb $D$ with $S_D=S$ is often called an $n$-comb (with $n=\mathrm{card}\,S$).

\begin{example}
The following has a positive 2-comb
\[\begin{picture}(6900,2250)(-300,-900)
\multiput(0,0)(4500,0){2}{\usebox{\rightbiedge}}
\put(0,0){\usebox{\gcircle}}
\put(1800,0){\usebox{\aone}}
\multiput(4500,0)(1800,0){2}{\usebox{\aone}}
\put(6300,600){\usebox{\tow}}
\multiput(1800,1350)(2700,0){2}{\line(0,-1){450}}
\put(1800,1350){\line(1,0){2700}}
\end{picture}\]
\end{example}

\subsubsection*{Tails}

\begin{definition}\label{def:tail}
Let $\S=(S^p,\sr,\A)$ be a spherical $R$-system. Let $\sigma\in\sr$ have support included in say $\{\alpha_1,\ldots,\alpha_n\}$, set of simple roots of an irreducible component of $R$. Let $\co'$ be a good distinguished subset of colors such that $\sr/\co'=\{\sigma\}$. Then $\sigma$ is called:
\begin{itemize}
\item tail of type $b(m)$, $1\leq m\leq n$, if $R$ is of type $\mathsf B_n$, $\sigma=\alpha_{n-m+1}+\ldots+\alpha_n$ and $\alpha_n\in S^p$ if $m>1$ (or $c(D^+_{\alpha_n},\sigma')=c(D^-_{\alpha_n},\sigma')$ for all $\sigma'\in\sr$ if $m=1$); 
\item tail of type $2b(m)$, $1\leq m\leq n$, if $R$ is of type $\mathsf B_n$ and $\sigma=2\alpha_{n-m+1}+\ldots+2\alpha_n$; 
\item tail of type $c(m)$, $2\leq m\leq n$, if $R$ is of type $\mathsf C_n$ and $\sigma=\alpha_{n-m+1}+2\alpha_{n-m+2}+\ldots+2\alpha_{n-1}+\alpha_n$;
\item tail of type $d(m)$, $2\leq m\leq n$, if $R$ is of type $\mathsf D_n$ and $\sigma=2\alpha_{n-m+1}+\ldots+2\alpha_{n-2}+\alpha_{n-1}+\alpha_n$.
\end{itemize}
\end{definition}

\begin{example}
\[\begin{picture}(27000,1800)
\put(300,900){
\put(0,0){\usebox{\shortam}}\put(3600,0){\usebox{\edge}}\put(5400,0){\usebox{\shortshortbm}}
}
\put(12000,900){\vector(1,0){3000}}
\put(15900,900){
\put(0,0){\usebox{\susp}}\put(3600,0){\usebox{\wcircle}}\put(3600,0){\usebox{\edge}}\put(5400,0){\usebox{\shortshortbm}}
}
\end{picture}\]
\[\begin{picture}(27000,1800)
\put(300,900){
\put(0,0){\usebox{\shortam}}\put(3600,0){\usebox{\edge}}\put(5400,0){\usebox{\shortshortbprimem}}
}
\put(12000,900){\vector(1,0){3000}}
\put(15900,900){
\put(0,0){\usebox{\susp}}\put(3600,0){\usebox{\wcircle}}\put(3600,0){\usebox{\edge}}\put(5400,0){\usebox{\shortshortbprimem}}
}
\end{picture}\]
\[\begin{picture}(23400,1800)
\put(300,900){
\put(0,0){\usebox{\atwo}}\put(1800,0){\usebox{\shortshortcm}}
}
\put(10200,900){\vector(1,0){3000}}
\put(14100,900){
\put(0,0){\usebox{\edge}}\put(1800,0){\usebox{\wcircle}}\put(1800,0){\usebox{\shortshortcm}}
}
\end{picture}\]
\[\begin{picture}(25800,2400)
\put(300,1200){
\put(0,0){\usebox{\shortam}}\put(3600,0){\usebox{\edge}}\put(5400,0){\usebox{\shortshortdm}}
}
\put(11400,1200){\vector(1,0){3000}}
\put(15300,1200){
\put(0,0){\usebox{\susp}}\put(3600,0){\usebox{\wcircle}}\put(3600,0){\usebox{\edge}}\put(5400,0){\usebox{\shortshortdm}}
}
\end{picture}\]
\end{example}

\begin{definition}\label{def:exl}
Let $\S=(S^p,\sr,\A)$ be a spherical $R$-system. Let $\widetilde{\sr}=\{\sigma_1,\sigma_2\}\subset\sr$ have support included in say $\{\alpha_1,\ldots,\alpha_n\}$, set of simple roots of an irreducible component of $R$. Let $\co'$ be a good distinguished subset of colors such that $\sr/\co'=\widetilde{\sr}$. Then $\widetilde{\sr}$ is called:
\begin{itemize}
\item tail of type $(aa,aa)$ if $R$ is of type $\mathsf E_6$ and $\widetilde{\sr}=\{\alpha_1+\alpha_6,\alpha_3+\alpha_5\}$; 
\item tail of type $(d3,d3)$ if $R$ is of type $\mathsf E_7$ and $\widetilde{\sr}=\{\alpha_2+2\alpha_4+\alpha_5,\alpha_5+2\alpha_6+\alpha_7\}$;
\item tail of type $(d5,d5)$ if $R$ is of type $\mathsf E_8$ and $\widetilde{\sr}=\{2\alpha_1+\alpha_2+2\alpha_3+2\alpha_4+\alpha_5,\alpha_2+\alpha_3+2\alpha_4+2\alpha_5+2\alpha_6\}$;
\item tail of type $(2a,2a)$ if $R$ is of type $\mathsf F_4$ and $\widetilde{\sr}=\{2\alpha_3,2\alpha_4\}$.
\end{itemize}
\end{definition}

\begin{example}
\[\begin{picture}(19800,3000)
\put(300,2100){
\put(0,0){\usebox{\dynkinesix}}\multiput(0,0)(1800,0){2}{\usebox{\wcircle}}\multiput(5400,0)(1800,0){2}{\usebox{\wcircle}}\multiput(0,900)(7200,0){2}{\line(0,-1){600}}\put(0,900){\line(1,0){7200}}\multiput(1800,600)(3600,0){2}{\line(0,-1){300}}\put(1800,600){\line(1,0){3600}}\put(3600,0){\usebox{\atwos}}
}
\put(8400,2100){\vector(1,0){3000}}
\put(12300,2100){
\put(0,0){\usebox{\dynkinesix}}\multiput(0,0)(1800,0){2}{\usebox{\wcircle}}\multiput(5400,0)(1800,0){2}{\usebox{\wcircle}}\multiput(0,900)(7200,0){2}{\line(0,-1){600}}\put(0,900){\line(1,0){7200}}\multiput(1800,600)(3600,0){2}{\line(0,-1){300}}\put(1800,600){\line(1,0){3600}}\put(3600,0){\usebox{\wcircle}}
}
\end{picture}\]
\[\begin{picture}(23400,2400)
\put(300,2100){
\put(0,0){\usebox{\dynkineseven}}\multiput(3600,0)(3600,0){2}{\usebox{\gcircle}}\put(0,0){\usebox{\atwo}}
}
\put(10200,2100){\vector(1,0){3000}}
\put(14100,2100){
\put(0,0){\usebox{\dynkineseven}}\multiput(3600,0)(3600,0){2}{\usebox{\gcircle}}\put(1800,0){\usebox{\wcircle}}
}
\end{picture}\]
\[\begin{picture}(27000,2400)
\put(300,2100){
\put(0,0){\usebox{\dynkineeight}}\multiput(0,0)(7200,0){2}{\usebox{\gcircle}}\put(9000,0){\usebox{\atwo}}
}
\put(12000,2100){\vector(1,0){3000}}
\put(15900,2100){
\put(0,0){\usebox{\dynkineeight}}\multiput(0,0)(7200,0){2}{\usebox{\gcircle}}\put(9000,0){\usebox{\wcircle}}
}
\end{picture}\]
\[\begin{picture}(16200,1800)
\put(300,900){
\put(0,0){\usebox{\dynkinffour}}\multiput(3600,0)(1800,0){2}{\usebox{\aprime}}\put(0,0){\usebox{\atwo}}
}
\put(6600,900){\vector(1,0){3000}}
\put(10500,900){
\put(0,0){\usebox{\dynkinffour}}\multiput(3600,0)(1800,0){2}{\usebox{\aprime}}\put(1800,0){\usebox{\wcircle}}
}
\end{picture}\]
\end{example}

\subsubsection*{Primitive spherical systems}

\begin{definition}\label{def:pss}
A spherical $R$-system is called primitive if it is cuspidal, not decomposable, without positive combs and without tails\footnote{Notice that in the previous papers (\cite{BP05,Br07}) this definition was less restrictive: tails were allowed.} (of type $b(m)$, $2b(m)$, $c(m)$, $d(m)$, $(aa,aa)$, $(d3,d3)$, $(d5,d5)$ or $(2a,2a)$).
\end{definition}

To present them all, we subdivide the primitive spherical systems into clans, which can be characterized as follows:
\begin{description}
\item{\bf clan \fR:} primitive spherical systems with a spherical root of type $aa$ or $2a$ without spherical roots with overlapping supports or with a spherical root of type $d(m)$ $m\geq3$;
\item{\bf clan \fS:} primitive spherical systems with only spherical roots of type $a(m)$ ($m\geq1$) or $b(2)$ (with $S^p=\emptyset$) without spherical roots with overlapping supports;
\item{\bf clan \fT:} primitive spherical systems with a spherical root of type $a(m)$ ($m\geq1$) whose support meets the support of another spherical root.
\end{description}

We give only the Luna diagram of the spherical systems. We use identifications between spherical roots as in Definition~\ref{def:sr}, namely when in a diagram there is a spherical root of type $a(m)$, $2b(m)$ or $d(m)$ without overlapping support then $m$ must be intended as $\geq1$, $\geq1$ or $\geq2$, respectively. In the diagrams of family \fc\ of the clan \fT\ there is always a spherical root of type $c(m)$, here $m$ must be intended as $\geq2$.

Furthermore, notice that in the following list primitive spherical $R$-systems are given up to external automorphism of $R$.

\begin{theorem}\label{thm:pss}
The clans $\fR$, $\fS$ and $\fT$, given below, contain all the primitive spherical systems of rank $>2$.
\end{theorem}

\pssclan{\fR}

\pss

$\a^{\faa}(p,p)$ 
\[\diagramaapp\]

$\b^{\faa}(p,p)$ 
\[\diagrambbpp\]

$\c^{\faa}(p,p)$ 
\[\diagramccpp\]

$\d^{\faa}(p,p)$ 
\[\diagramddpp\]

$\e^{\faa}(p,p)$ 
\[
\]

\subsection{Primitive positive 1-combs}

\begin{definition}
A positive 1-comb of a spherical $R$-system $\S$ is called primitive if $\S$ is cuspidal, not decomposable and without tails.
\end{definition}

\begin{theorem}\label{thm:ppc}
The families $\fp$ and $\fq$, given below, contain all the spherical systems of rank $>2$ with a primitive positive 1-comb. 
\end{theorem}

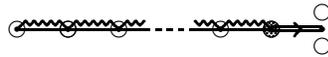
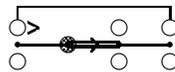
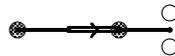
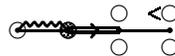
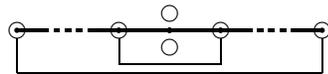
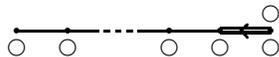
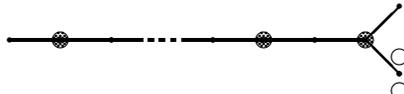
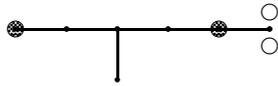
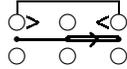
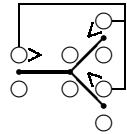
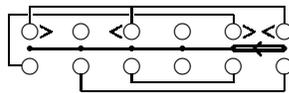
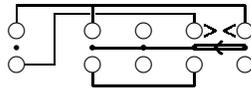
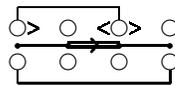
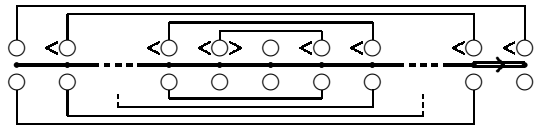
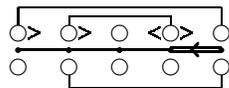
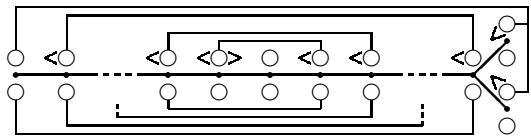
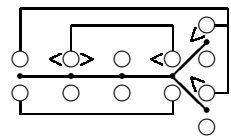
\renewcommand{\pssclanlabel}{\ppc}
\setcounter{pssnumber}{0}

\pssfamily{Family \fp}

\pss \label{onecombmodelb}
\[%
\]

\subsection{Minimal quotients}\label{subsec:qpss}

The distinguished subsets of colors of a given spherical system can easily be determined by looking at the integer matrix of the corresponding Cartan pairing. Often, for a given subset of colors $\co'$ either $\sum_{D\in\co'}D$ is positive (then $\co'$ is distinguished) or there exists some $\sr'\subset\sr$ such that $c(D,\sum_{\sigma\in\sr'} \sigma)\leq0$, not always zero, for all $D\in\co'$ (then $\co'$ is not distinguished).

Recall that the defect $\mathrm{d}(\S)$ of a spherical system $\S=(S^p,\sr,\A)$ with set of colors $\co$ is $\mathrm d(\S)=\mathrm{card}\,\co-\mathrm{card}\,\sr$.

We compute all the minimal quotients of primitive spherical systems and of spherical systems with a primitive positive 1-comb (of rank $>2$). For the clan \fR, these are particularly easy to describe (see below); for the remaining cases, we give all the corresponding diagrams in Appendix~\ref{app:qpss}.

\begin{remark}\label{rem:pgood}
All distinguished subsets of colors of primitive spherical systems or of spherical systems with a primitive positive 1-comb are good.
\end{remark}

Here are some complementary remarks on the structure of some classes of primitive spherical systems with special attention to their quotients.

First, we remark that all the primitive spherical systems have defect $\leq2$.

\subsubsection*{Clan \fR}

All the members of the clan \fR\ but the cases \fR-\ref{aapplusqplusp}, \fR-\ref{con}, \fR-\ref{dcneven}, \fR-\ref{dcnodd} and \fR-\ref{efseven} have defect 0 and a unique minimal distinguished subset of colors, the full set of colors $\co$, which is homogeneous and $\S/\co=(S,\emptyset,\emptyset)$. In the other (above listed) cases the defect is 1 and there are two colors $D_1,D_2$ such that $c(D_1,\sigma)=c(D_2,\sigma)$ for all $\sigma\in\sr$, and there are exactly two minimal distinguished subsets of colors, namely $\co\setminus\{D_1\}$ and $\co\setminus\{D_2\}$, which are homogeneous and such that $S^p/(\co\setminus\{D_i\})=\{\alpha\}$ with $D_i\in\co(\alpha)$. 

\subsubsection*{Clan \fS}

All the members of the family \fx\ have defect 0. 
All the members of the families \fy, \fz, \fu, \fv, \fw\ have defect 1. 

\subsubsection*{Clan \fT}

\paragraph{Family \fma.}

All the members of the clan \fT\ with only spherical roots of type $a(m)$ have the same rank, the same number of colors and the same Cartan pairing of some corresponding members of the clan \fS, therefore such spherical systems have corresponding minimal quotients also. For brevity let us here denote such correspondence by an equivalence symbol $\sim$. In case \fT-\ref{model} we have: 
\begin{itemize}
\item[] $\a^\fm(2p+1)\sim \a^\fy(p,p)$, $\a^\fm(2p)\sim \a^\fy(p-1,p)$, 
\item[] $\d^\fm(2p+2)\sim \ad^\fy(p,p+1)$, $\d^\fm(2p+1)\sim \ad^\fy(p-1,p+1)$, 
\item[] $\e^\fm(8) \sim \d^\fv(7)$, $\e^\fm(7) \sim \a^\fu(6)$, $\e^\fm(6) \sim \a^\fu(5)$;
\end{itemize} 
and furthermore:
\begin{itemize}
\item[] \fT-\ref{acastnplusaone} $\sim$ \fS-\ref{aypplusqplusp} if rank odd, $\sim$ \fS-\ref{aypplusqpluspminusone} if rank even ($\sim$ \fS-\ref{axonepluspplusone} if rank 3),
\item[] \fT-\ref{acastnplusaonee} $\sim$ \fS-\ref{aythreethreeplusapiiid} if rank 7, $\sim$ \fS-\ref{aythreetwoplusapiiid} if rank 6, $\sim$ \fS-\ref{aytwotwoplusapbelow} if rank 5,
\item[] \fT-\ref{acastnplustwocombe} $\sim$ \fS-\ref{aythreetwoplustwocombi} if rank 7, $\sim$ \fS-\ref{aytwotwoplustwocombi} if rank 6, $\sim$ \fS-\ref{aytwooneplustwocombi} if rank 5, 
\item[] \fT-\ref{acastplustwocombd} $\sim$ \fS-\ref{twocombplustwocomb},
\item[] \fT-\ref{acastplusax} $\sim$ \fS-\ref{axplustwocomb},
\item[] \fT-\ref{acastthreeplusaytwotwo} ($\sim$ \fT-\ref{acastnplustwocombe} of rank 6) $\sim$ \fS-\ref{aytwotwoplustwocombi},
\item[] \fT-\ref{acastthreeplusayonetwo} ($\sim$ \fT-\ref{acastnplustwocombe} of rank 5) $\sim$ \fS-\ref{aytwooneplustwocombi},
\item[] \fT-\ref{dsastfour} $\sim \a^\fx(1,1,1)$,
\item[] \fT-\ref{dsnplusaone} ($\sim$ \fT-\ref{acastnplusaone} of rank 3) $\sim$ \fS-\ref{axonepluspplusone},
\item[] \fT-\ref{dsastfourplusaoneplusaone} $\sim$ \fS-\ref{axonepluspplusoneplusqplusone},
\item[] \fT-\ref{dsastfourplusaone} $\sim$ \fS-\ref{axonepluspplusoneone},
\item[] \fT-\ref{dsnplusonecomb} ($\sim$ \fT-\ref{acastplustwocombd}) $\sim$ \fS-\ref{twocombplustwocomb},
\item[] \fT-\ref{dsastfourplusonecomb} ($\sim$ \fT-\ref{acastplusax}) $\sim$ \fS-\ref{axplustwocomb}.
\end{itemize} 

\paragraph{Family \fc.}

The case \fT-\ref{genmodelc} is a generalization of $\c^\fm(n)$, with similar quotients. Let us restrict our attention to the remaining cases.

The case \fT-\ref{aaplusaonepluscm} below (on the right hand side), after ``collapsing'' the spherical root of type $c(m)$, corresponds to the primitive spherical system below (on the left hand side), the minimal quotients of the former are in correspondence with the minimal quotients of the latter:
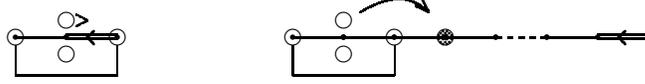
\[\begin{picture}(4200,2250)(-300,-1350)\put(0,0){\usebox{\dynkincthree}}\multiput(0,0)(3600,0){2}{\usebox{\wcircle}}\multiput(0,-1350)(3600,0){2}{\line(0,1){1050}}\put(0,-1350){\line(1,0){3600}}\put(1800,0){\usebox{\aone}}\put(1800,600){\usebox{\toe}}\end{picture}
\qquad\qquad\qquad
\begin{picture}(12900,2250)
\multiput(300,1350)(1800,0){2}{\usebox{\edge}}
\multiput(300,1350)(3600,0){2}{\circle{600}}
\multiput(300,0)(3600,0){2}{\line(0,1){1050}}
\put(300,0){\line(1,0){3600}}
\put(2100,1350){\usebox{\aone}}
\put(3900,1350){\usebox{\pluscsecondm}}
\put(2100,1950){\usebox{\tobe}}
\end{picture}\]
The same relation holds for the other members of the family \fc\ and the following primitive spherical systems, respectively: \fS-\ref{cxonepluoneplusoneii}, \fS-\ref{bcx} namely $\bc^\fx(2,3)$, \fS-\ref{cu} namely $\c^\fu(5)$, \fS-\ref{axoneoneoneplusapplusaonec}, \fS-\ref{axoneoneoneplusaonec} and \fS-\ref{aytwooneplusaonec}. 

\subsection{General spherical systems}

A general spherical system can be obtained from primitive spherical systems by inducing (see Definition~\ref{def:ind}), composing, joining positive 1-combs and adding tails (see below).

\subsubsection*{Composing}

\begin{proposition}
Let $\S_1=(S^p_1,\sr_1,\A_1)$ and $\S_2=(S^p_2,\sr_2,\A_2)$ be spherical $R$-systems with set of colors $\co_1$ and $\co_2$, respectively. Let $\co''\subset\co_1$ and $\co'\subset\co_2$ be good distinguished subsets with $\S_1/\co''=\S_2/\co'$, say equal to $(S^p_{1,2},\sr_{1,2},\A_{1,2})$, such that
\begin{itemize}
\item $(S^p_{1,2}\setminus S^p_1)\perp (S^p_{1,2}\setminus S^p_2)$ and
\item $(\sr_1\setminus\sr_{1,2})\cap(\sr_2\setminus\sr_{1,2})=\emptyset$.
\end{itemize}
Set
\begin{itemize}
\item $S^p=S^p_1\cap S^p_2$,
\item $\sr=(\sr_1\setminus\sr_{1,2})\cup(\sr_2\setminus\sr_{1,2})\cup(\sr_1\cap\sr_2)$,
\item $\A=\A_1\cup\A_2$ (identifying $\A_1/\co''\subset \A_1$ with $\A_2/\co'\subset\A_2$).   
\end{itemize}
Then 
\begin{itemize}
\item $\S=(S^p,\sr,\A)$ is a well defined spherical $R$-system, with set of colors $\co=\co_1\cup\co_2$ (identifying $\co_1\setminus\co''$ with $\co_2\setminus\co'$), 
\item $\co'$ and $\co''$ decompose $\S$,
\item $\S/\co'=\S_1$ and $\S/\co''=\S_2$.
\end{itemize}
\end{proposition}

\begin{proof}
If $D\in\co'$ then $c(D,\sigma)=0$ for all $\sigma\in\sr_1\setminus\sr_{1,2}$. Moreover, for all $\alpha\in S\cap\sr_2\setminus\sr_{1,2}$ one has $\A_2(\alpha)\cap\co'\neq\emptyset$ (and analogously the symmetric implication holds). 

In particular this shows that the restricted Cartan pairing of $\S$ is well defined. Indeed, let $D$ be in $\A(\alpha)$. If $\alpha\in\sr_1\cap\sr_2\subset\sr_{1,2}$ then $\A(\alpha)$ can be identified with $\A_1(\alpha)=\A_2(\alpha)=\A_{1,2}(\alpha)$, thus set $c(D,\sigma)=c_i(D,\sigma)$ if $\sigma\in\sr_i$, $i=1,2$, and notice that $c_1(D,\sigma)=c_2(D,\sigma)$ if $\sigma\in\sr_1\cap\sr_2\subset\sr_{1,2}$. If $\alpha\in\sr_2\setminus\sr_{1,2}$ then, for all $\sigma\in\sr_1\setminus\sr_{1,2}$, $c(D,\sigma)=0$ if $D\in\co'$ or $c(D,\sigma)=\langle\alpha^\vee,\sigma\rangle$ if $D\not\in\co'$ (notice that if $\A_2(\alpha)\subset\co'$ then $\alpha\perp\sigma$ for all $\sigma\in\sr_1\setminus\sr_{1,2}$); finally, $c(D,\sigma)=c_2(D,\sigma)$ for all $\sigma\in\sr\cap\sr_2$.

Analogously, the axioms of spherical system for $\S$ and the rest of the statement follow.
\end{proof}

We say that the spherical system $\S$ as above is obtained by composing the spherical systems $\S_1$ and $\S_2$.

\subsubsection*{Joining positive combs}

\begin{remark}\label{rem:join}
Let $\S=(S^p,\sr,\A)$ be a spherical $R$-system with $D_1,\ldots,D_m$ positive combs with disjoint subsets $S_{D_1},\ldots,S_{D_m}$. Let $S_1,\ldots,S_k$ be such that $S_1\sqcup\ldots\sqcup S_k=S_{D_1}\sqcup\ldots\sqcup S_{D_m}$. The spherical $R$-system $(S^p,\sr,\A')$ where $\A'$ is obtained from $\A$ by replacing the positive combs $D_1,\ldots,D_m$ by other positive combs, say $D'_1,\ldots,D'_k$, with $S_{D'_i}=S_i$, $1\leq i\leq k$, is well defined. In particular, if $D_1,\ldots,D_m$ are positive 1-combs, it is always possible to join them, namely replace them by a unique positive $m$-comb $D'$ with $S_{D'}=S_{D_1}\sqcup\ldots\sqcup S_{D_m}$.
\end{remark}

\subsubsection*{Adding tails}

A color $D$ is called free if there exists at most one spherical root $\sigma$ with $c(D,\sigma)>0$.

\begin{proposition} Let $\S$ be a spherical $R$-system with a tail $\sigma\in\sr$ and notation as in Definition~\ref{def:tail}.
\begin{enumerate}
\item If $\sigma$ is of type $b(m)$, $2b(m)$ or $d(m)$, then there exist (at most) one spherical root $\sigma'$ with $\alpha_{n-m}\in\mathrm{supp}\,\sigma'$ and one free color $D'$ such that $c(D',\sigma)<0$ and $c(D',\sigma')=1$. The spherical root $\sigma'$ is of type $a(k)$, $k\geq1$. 
\item If $\sigma$ is of type $c(m)$, then there exist (at most) one spherical root $\sigma^{(0)}$ with $\alpha_{n-m+1}\in\mathrm{supp}\,\sigma^{(0)}$, one spherical root $\sigma'$ with $\alpha_{n-m}\in\mathrm{supp}\,\sigma'$ and one free color $D'$ such that $c(D',\sigma)<0$ and $c(D',\sigma')=1$. The spherical root $\sigma^{(0)}$ is either of type $a(1)$, of type $aa$, of type $2a$ or of type $a(2)$. In the case $a(2)$ $\sigma^{(0)}=\sigma'$, in the case $2a$ there is no $\sigma'$, otherwise $\alpha_{m-n+1}\not\in\mathrm{supp}\,\sigma'$ and $\sigma'$ is of type $a(k)$, $k\geq1$. 
\end{enumerate}
\end{proposition}

\begin{proof}
The less evident case is that of type $b(1)$, with $\sigma'$ of type $a(1)$, where to check the last assertion one has to notice that if there exists a simple spherical root $\sigma''\neq\sigma'$ with $c(D',\sigma'')=1$ then there exists a color $D\in\co(\sigma)\cap\co(\sigma'')$: hence the subset $\co'$ of Definition~\ref{def:tail} can not exist.
\end{proof}

\begin{remark} \label{rem:tail}
Let $\S'=((S')^p,\sr',\A')$ be a spherical $R'$-system. 
\begin{enumerate}
\item If $n>m$, let $\alpha_1,\ldots,\alpha_{n-m}$ be the simple roots of an irreducible component of type $\mathsf A$ of $R'$. If there exists a spherical root $\sigma'$ with $\alpha_{n-m}\in\mathrm{supp}\,\sigma'$, assume (i) it is of type $a(k)$, $k\geq1$, and (ii) there exists a free color $D'\in\co(\alpha_{n-m})$ with $c(D',\sigma')=1$ ((ii) follows from (i) if $k>1$). Then is well defined the spherical $R$-system $\S$ with a tail $\sigma$ of type $b(m)$, $2b(m)$ or $d(m)$ (with notation as in Definition~\ref{def:tail}, $\mathrm{supp}\,\sigma=\{\alpha_{n-m+1},\ldots,\alpha_n\}$ and $\sr=\sr'\cup\{\sigma\}$) and with $D'$ unique color such that $c(D',\sigma)<0$. The fact that $\sigma$ is a tail, i.e.\ that there exists a good distinguished subset of colors $\co$ as in Definition~\ref{def:tail} can be checked (see Lemma~\ref{lem:A}) on the primitive spherical systems. 
\item If $n>m$, let $\alpha_1,\ldots,\alpha_{n-m+1}$ be the simple roots of an irreducible component of type $\mathsf A$ of $R'$. Assume there exists a spherical root $\sigma^{(0)}$ with $\alpha_{n-m+1}\in\mathrm{supp}\,\sigma^{(0)}$. If there exists a spherical root $\sigma'$ with $\alpha_{n-m}\in\mathrm{supp}\,\sigma'$, assume (i) it is of type $a(k)$, $k\geq1$, and (ii) there exists a free color $D'\in\co(\alpha_{n-m})$ with $c(D',\sigma')=1$. Then is well defined the spherical $R$-system $\S$ with a spherical root $\sigma$ of type $c(m)$ (with notation as in Definition~\ref{def:tail}, $\mathrm{supp}\,\sigma=\{\alpha_{n-m+1},\ldots,\alpha_n\}$ and $\sr=\sr'\cup\{\sigma\}$) and with $D'$ unique color such that $c(D',\sigma)<0$. The spherical root $\sigma$ is not necessarily a tail, i.e.\ a good distinguished subset of colors $\co$ as in Definition~\ref{def:tail} does not necessarily exist. If $n=m$, there exists no $\sigma'$ (and no $D'$) but the spherical $R$-system with a spherical root $\sigma$ of type $c(m)$ (as above, but with no color $D'$ such that $c(D',\sigma)<0$) is still well defined and $\sigma$ is always a tail.  
\end{enumerate}
\end{remark}

The behavior of tails of exceptional type is similar to that of tails of type $2b(m)$ or $d(m)$.

\begin{proposition} Let $\S$ be a spherical $R$-system with a tail $\widetilde{\sr}\subset\sr$ and notation as above. Then there exist (at most) one spherical root $\sigma'$ not orthogonal to $\widetilde{\sr}$ and one free color $D'$ such that $c(D',\sigma_i)<0$ for some $i=1,2$ and $c(D',\sigma')=1$. The spherical root $\sigma'$ is of type $a(k)$, $1\leq k\leq2$.
\end{proposition}

\begin{remark} 
Let $\S'=((S')^p,\sr',\A')$ be a spherical $R'$-system. 
Let $\beta_1,\beta_2$ be the simple roots of an irreducible component of type $\mathsf A$ of $R'$. If there exists a spherical root $\sigma'$ with $\beta_2\in\mathrm{supp}\,\sigma'$, assume (i) it is of type $a(k)$, $1\leq k\leq2$, and (ii) there exists a free color $D'\in\co(\beta_2)$ with $c(D',\sigma')=1$ ((ii) follows from (i) if $k>1$). Then is well defined the spherical $R$-system $\S$ with a tail $\widetilde{\sr}$ of type $(aa,aa)$, $(d3,d3)$, $(d5,d5)$ or $(2a,2a)$ (with notation as in Definition~\ref{def:exl}, $\sr=\sr'\cup\widetilde{\sr}$) and with $D'$ unique color such that $c(D',\sigma_i)<0$, for some $i=1,2$. 
\end{remark}

We say that the spherical system $\S$ as above is obtained by adding a tail to the spherical system $\S'$.

\subsection{Color-adjacent spherical roots and gluing}\label{subsec:col}

To prove here Theorems \ref{thm:pss} and \ref{thm:ppc} we introduce several technical combinatorial notions. Although they have no clear geometric or group theoretic counterpart on the side of wonderful varieties (and other slightly different notions could have been used), they help understanding and describing more explicitly the combinatorial structure of a general spherical system.

\subsubsection*{Color-connected spherical systems}

For simplicity, let us assume for the moment that 
\begin{equation}\label{eqn:nsc} 
\parbox{10.6cm}{if $\co(\alpha)\neq\emptyset$ then $\alpha$ lies in the support of (at most) one spherical root.}
\end{equation}

Two spherical roots $\sigma_1,\sigma_2$ will be called color-adjacent if, for all colors $D\in\co(\alpha)$ for some $\alpha\in\mathrm{supp}\,\sigma_i$, $c(D,\sigma_j)\neq0$, for $i\neq j$. In Appendix~\ref{app:low} one can find all the possible pairs of such color-adjacent spherical roots. Two spherical roots will be called color-connected if satisfy the above relation extended by transitivity. For short, a spherical system will be called color-connected if spherical roots are pairwise color-connected.

\begin{lemma}
A cuspidal color-connected spherical system (under the condition \eqref{eqn:nsc} above) of rank $>2$ 
\begin{itemize}
\item[-] is primitive and is either a member of the clan \fR, or a member of the families $\fx$, $\fy$, $\fz$, $\fu$, $\fv$, $\fw$ of the clan \fS\ 
\item[-] or is an $n$-comb.
\end{itemize}
\end{lemma}

\begin{proof}
It is enough to proceed recursively: assume we have such a color-connected spherical system, then see if it can be the localization of another color-connected spherical system with a spherical root more.

Let us start with a spherical root of type $aa$, it can be color-adjacent to a spherical root of type either $aa$, $2a$, $a(p)$ ($p\geq1$) or $b(2)$. If it is color-adjacent to a spherical root of type $a(1)$, then the latter cannot be color-adjacent to another spherical root of same type $a(1)$ by the axiom (A1). Then we can go on recursively but no other types of spherical roots can occur and we can find all the cuspidal color-connected spherical systems with a spherical root of type $aa$: the rank $>2$ cases are members of the clan \fR.

Let us start with a spherical root of type $d(m)$. It can be color-connected to a spherical root of type either: $d(3)$, $2a$, $a(1)$, $a(3)$, $c(p)$, $2b(2)$ (if $m=3$), $2a$, $a(1)$, $a(5)$ (if $m=4$), $d(5)$, $2a$, $a(1)$ (if $m=5$) or $2a$, $a(1)$ (if $m>5$). Again a spherical root of type $a(1)$ here cannot be color-adjacent to another spherical root of same type $a(1)$. We get a list of rank $>2$ cases which is included in the clan \fR.

Let us start with a spherical root of type $2a$. It can be color-connected to a spherical root of type $2a$, $a(1)$, $2b(p)$ ($p\geq2$). Roots of type $aa$ and $d(p)$ have already been considered. Similarly, we get a list of rank $>2$ cases which is included in the clan \fR.

We get no new cases by starting with a spherical root of type $2b(m)$.

If we start with spherical roots of type different from those considered above and different from $a(1)$ we get only rank 1 or rank 2 cases.

Let us restrict to the cases with only spherical roots of type $a(1)$. Two such roots $\alpha_1,\alpha_2$ are color-adjacent if and only if $\co(\alpha_1)\cap\co(\alpha_2)\neq\emptyset$. Here it is possible to proceed recursively as above (applying the axioms (A1) and (A2)) and get all the cases, but they are many and the procedure is very long. There are the combs and all the cases listed in the families $\fx$, $\fy$, $\fz$, $\fu$, $\fv$, $\fw$ of the clan \fS. Alternatively, one can list all the cases by computer up to rank 9. A color-connected spherical system with $\sr=S$ of rank 9 is $\bd^\fx(4,5)$, $\a^\fy(4,5)$, $\ab^\fy(4,5)$, $\ad^\fy(4,5)$, $\bd^\fy(4,5)$ or a $9\comb$. Then it is possible to start from there to prove by induction that the only color-connected spherical systems with $\sr=S$ and rank $\geq9$ (that are not combs) are:
\begin{itemize}
\item $\bd^\fx(p-1,p)$, $\a^\fy(p-1,p)$, $\ab^\fy(p-1,p)$, $\ad^\fy(p-1,p)$, $\bd^\fy(p-1,p)$ (with odd rank $2p-1$), and
\item $\bd^\fx(p,p)$, $\a^\fy(p,p)$, $\ab^\fy(p,p)$, $\b^\fy(p,p)$, $\d^\fy(p,p)$, $\ad^\fy(p-1,p+1)$, $\a^\fy(2p)$ (with even rank $2p$).
\end{itemize}
\end{proof}

\subsubsection*{Plugs}

Let us consider the following spherical systems:
\begin{equation}\label{eqn:plug}
\parbox{10.6cm}{
\[\begin{picture}(9600,2250)(-300,-900)
\multiput(0,0)(7200,0){2}{\usebox{\edge}}
\multiput(0,0)(9000,0){2}{\usebox{\aone}}
\put(1800,0){\usebox{\mediumam}}
\put(0,600){\usebox{\toe}}
\put(9000,600){\usebox{\tow}}
\multiput(0,1350)(9000,0){2}{\line(0,-1){450}}
\put(0,1350){\line(1,0){9000}}
\end{picture}
\qquad\qquad\qquad
\begin{picture}(9600,2250)(-300,-900)
\put(0,0){\usebox{\edge}}
\put(7200,0){\usebox{\rightbiedge}}
\multiput(0,0)(9000,0){2}{\usebox{\aone}}
\put(1800,0){\usebox{\mediumam}}
\put(0,600){\usebox{\toe}}
\put(9000,600){\usebox{\tow}}
\multiput(0,1350)(9000,0){2}{\line(0,-1){450}}
\put(0,1350){\line(1,0){9000}}
\end{picture}
\]
\[
\begin{picture}(9600,2250)(-300,-900)
\put(0,0){\usebox{\edge}}
\put(7200,0){\usebox{\leftbiedge}}
\multiput(0,0)(9000,0){2}{\usebox{\aone}}
\put(1800,0){\usebox{\mediumam}}
\put(0,600){\usebox{\toe}}
\put(9000,600){\usebox{\tow}}
\multiput(0,1350)(9000,0){2}{\line(0,-1){450}}
\put(0,1350){\line(1,0){9000}}
\end{picture}
\qquad\qquad
\begin{picture}(4200,2250)(-300,-900)
\put(0,0){\usebox{\dynkincthree}}
\multiput(0,0)(1800,0){3}{\usebox{\aone}}
\multiput(0,600)(1800,0){2}{\usebox{\toe}}
\put(3600,600){\usebox{\tow}}
\multiput(0,1350)(3600,0){2}{\line(0,-1){450}}
\put(0,1350){\line(1,0){3600}}
\end{picture}
\qquad\qquad
\begin{picture}(6000,2250)(-300,-900)
\put(0,0){\usebox{\dynkinf}}
\multiput(0,0)(5400,0){2}{\usebox{\aone}}
\put(1800,0){\usebox{\bsecondtwo}}
\put(0,600){\usebox{\toe}}
\put(5400,600){\usebox{\tow}}
\multiput(0,1350)(5400,0){2}{\line(0,-1){450}}
\put(0,1350){\line(1,0){5400}}
\end{picture}
\]
}
\end{equation}

We say that two spherical roots $\sigma_1,\sigma_2$ of a spherical $R$-system $\S$ lie in a plug if there exists $S'\subset S$ with $\mathrm{supp}\,\sigma_1\cup\mathrm{supp}\,\sigma_2\subset S'$ such that the corresponding localization equals one of the above systems.

We easily obtain the following

\begin{lemma}\label{lem:plug}
A cuspidal spherical system (under the condition \eqref{eqn:nsc}) of rank $>2$ such that every pair of not color-adjacent spherical roots lies in a plug as above is primitive and is a member of the family \fp\ of the clan \fS.
\end{lemma}

\subsubsection*{Weakly-color-connected spherical systems}

Let us abandon the condition \eqref{eqn:nsc}. Notice that the following spherical system has the same rank, the same defect and the same Cartan pairing of the first spherical system of \eqref{eqn:plug}: it can thus be considered as a generalized plug, as well.
\begin{equation}\label{eqn:pl2}
\parbox{10.6cm}{
\[\begin{picture}(9000,3000)(-300,-1500)
\put(0,0){\usebox{\edge}}
\put(0,0){\usebox{\aone}}
\put(0,600){\usebox{\toe}}
\put(1800,0){\usebox{\amne}}
\put(1800,0){\usebox{\amse}}
\end{picture}\]
}
\end{equation}

Generalizing the notion of plugs, we define weak plugs: we say that two spherical roots $\sigma_1,\sigma_2$ of a spherical $R$-system $\S$ lie in a weak plug if there exists $S'\subset S$ with $\mathrm{supp}\,\sigma_1\cup\mathrm{supp}\,\sigma_2\subset S'$ such that the corresponding localization satisfies the following: 
\begin{equation}\label{eqn:wpl}
\parbox{10.6cm}{
$\sr'=\{\sigma_1,\sigma_2,\sigma_3\}$ and $\sigma_1,\sigma_2$ are color-connected or have overlapping support and there are two colors $D_1,D_2$ such that $D_i\in\co(\alpha_i)$ with $\alpha_i\in\mathrm{supp}\,\sigma_3$ and $c(D_i,\sigma_i)\neq0$, for $i=1,2$.
}
\end{equation}

Plugs are weak plugs. A particular weak plug is the following:
\[
\begin{picture}(11100,2250)(-300,-900)
\put(0,0){\usebox{\edge}}
\multiput(0,0)(1800,0){2}{\usebox{\aone}}
\put(1800,0){\usebox{\shortcm}}
\put(0,600){\usebox{\tobe}}
\end{picture}
\]

We say that two spherical roots are weakly-color-adjacent if are color-adjacent or have overlapping support or lie in a weak plug; we say that two spherical roots are weakly-color-connected if satisfy the same relation extended by transitivity; for short, we say that a spherical system is weakly-color-connected if its spherical roots are pairwise weakly-color-connected.

\begin{lemma}
A cuspidal weakly-color-connected spherical system 
\begin{itemize}
\item[-] is primitive and is a member of the clan \fR, of the clan \fS\ (but not of the family \fq), or of the clan \fT,
\item[-] or has a tail of type $c(m)$,
\item[-] or has a positive $n$-comb: if we here assume in addition that $n=1$, then the positive 1-comb is primitive and the spherical system is a member of the family \fp.
\end{itemize}
\end{lemma} 

\begin{proof}
{\it Step 1.} We keep on assuming color-connectedness but weaken condition \eqref{eqn:nsc} and assume that 
\begin{equation}\label{eqn:co2}
\parbox{10.6cm}{if $\co(\alpha)\neq\emptyset$ and $\alpha\in\mathrm{supp}\,\sigma_1\cap\mathrm{supp}\,\sigma_2$ then at least one of the two spherical roots $\sigma_1,\sigma_2$ is of type $a(m)$ with $m>1$.}
\end{equation}

We must start with a spherical root of type $a(m)$, $m>1$, and proceed recursively as above, but notice that now we are allowing only the following two further possibilities of pairs of color-adjacent spherical roots:
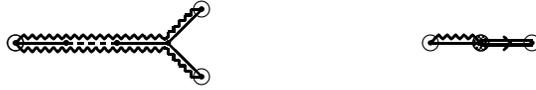
\[\diagramdsn\qquad\qquad\qquad\qquad\begin{picture}(4200,3000)(-300,-1500)\put(0,0){\usebox{\atwo}}\put(1800,0){\usebox{\bsecondtwo}}\end{picture}\]
In particular, notice that in the following case the two spherical roots are not color-adjacent:
\[\begin{picture}(11100,600)(-300,-300)\put(0,0){\usebox{\atwo}}\put(1800,0){\usebox{\shortcm}}\end{picture}\] 
Only spherical roots of type $a(m)$ (with $m>1$) and $b(2)$ (with support not meeting $S^p$) are involved. We obtain only members of \fT-1 (of types $\mathsf A$, $\mathsf B$, $\mathsf D$, $\mathsf E$) and the case \fT-\ref{dsastfour}.

{\it Step 2.} We allow plugs and generalized plugs \eqref{eqn:pl2}. We obtain all the primitive spherical systems of the family \fp\ of the clan \fS\ (as in Lemma~\ref{lem:plug}) and of the family \fma\ of the clan \fT. 

{\it Step 3.}
We go on allowing the following as localizations
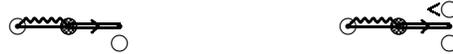
\[\begin{picture}(4200,1800)(-300,-900)\put(0,0){\usebox{\atwo}}\put(1800,0){\usebox{\btwo}}\put(3600,0){\usebox{\aprime}}\end{picture}\qquad\qquad\qquad\qquad\begin{picture}(4200,1800)(-300,-900)\put(0,0){\usebox{\atwo}}\put(1800,0){\usebox{\btwo}}\put(3600,0){\usebox{\aone}}\put(3600,600){\usebox{\tow}}\end{picture}\]
and we get the cases \fT-\ref{bcn}, \fT-\ref{bcn}' and \fT-\ref{fbcn}.

Allowing the following as localizations we get no new rank $>2$ cases (without positive $n$-combs with $n>1$).
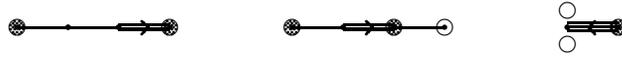
\[\begin{picture}(6000,1800)(-300,-900)\put(0,0){\usebox{\dynkinbfour}}\multiput(0,0)(5400,0){2}{\usebox{\gcircle}}\end{picture}
\qquad\qquad
\begin{picture}(6000,1800)(-300,-900)\put(0,0){\usebox{\dynkinf}}\multiput(0,0)(3600,0){2}{\usebox{\gcircle}}\put(5400,0){\usebox{\wcircle}}\end{picture}
\qquad\qquad
\begin{picture}(2400,1800)(-300,-900)\put(0,0){\usebox{\aone}}\put(0,0){\usebox{\lefttriedge}}\put(1800,0){\usebox{\gcircle}}\end{picture}
\]

{\it Step 4.} We include all pairs of spherical roots with overlapping support, namely also the following (with a root of type $c(m)$ with $m\geq2$):
\[\begin{array}{c@{\hspace{2cm}}c}
\begin{picture}(12000,1800)(-300,-900)\put(0,0){\usebox{\vertex}}\multiput(0,0)(2700,0){2}{\usebox{\wcircle}}\multiput(0,-900)(2700,0){2}{\line(0,1){600}}\put(0,-900){\line(1,0){2700}}\put(2700,0){\usebox{\shortcm}}\end{picture}
&
\begin{picture}(9300,1800)(-300,-900)\put(0,0){\usebox{\aprime}}\put(0,0){\usebox{\shortcm}}\end{picture}
\\
\begin{picture}(11100,2400)(-300,-900)\put(0,0){\usebox{\atwo}}\put(1800,0){\usebox{\shortcm}}\end{picture}
&
\begin{picture}(9300,2400)(-300,-900)\put(0,0){\usebox{\aone}}\put(0,0){\usebox{\shortcm}}\end{picture}
\end{array}\]
First, we get the cases \fT-\ref{azthreeoneplusbtwo}, \fT-\ref{aytwooneplusbtwob}, \fT-\ref{aytwooneplusbtwof}, \ppc-\ref{onecombmodelb}, \ppc-\ref{onecomboverfavoritei} and \ppc-\ref{onecombfavoriteii}. The rest are all cases (with a spherical root $\sigma$ of type $c(m)$) that can be constructed as in Remark~\ref{rem:tail}(2). These latter cases are many, but easy to construct and most of them actually have a tail of type $c(m)$. In the notation of Remark~\ref{rem:tail}, they have a tail when the spherical system $\S'$ has a homogeneous distinguished subset of colors that does not contain the color $D'$ (and this can be checked case-by-case). The only cases without tail are \fT-\ref{genmodelc} (if rank $>2$), \fT-\ref{abxpluscq} and \fT-\ref{au5pluscq}.

{\it Step 5.} Finally, we include weak plugs and obtain the rest of the clan \fT\ (namely some members of \fT-\ref{model}, the family \fmb\ and the rest of the family \fc) and the rest of the family \fp\ of spherical systems with a primitive positive 1-comb (namely only the case \ppc-\ref{onecombfvar}).
\end{proof}

\subsubsection*{Gluing}

Let $\S=(S^p,\sr,\A)$ be a cuspidal spherical $R$-system. Let $S'\subset S$ be such that $\S'=((S')^p,\sr',\A')$, the spherical system obtained from $\S$ by localization on $S'$, is weakly-color-connected. Then $\S'$ is called weakly-color-connected component of $\S$ if $S'$ is maximal with this property. The supports of the sets of spherical roots of the weakly-color-connected components are clearly disjoint, but not necessarily pairwise orthogonal. We will say that a cuspidal spherical $R$-system is obtained by gluing its weakly-color-connected components. 

A spherical $R$-system is not uniquely determined by its weakly-color-connected components, see for example
\[\begin{picture}(2400,1800)(-300,-900)\put(0,0){\usebox{\rightbiedge}}\multiput(0,0)(1800,0){2}{\usebox{\aone}}\end{picture}\qquad\qquad\qquad\qquad\begin{picture}(2400,1800)(-300,-900)\put(0,0){\usebox{\rightbiedge}}\multiput(0,0)(1800,0){2}{\usebox{\aone}}\put(1800,600){\usebox{\tow}}\end{picture}\]

\begin{remark}\label{rem:glue}
If $\S^{(1)},\S^{(2)}$ are two weakly-color-connected components, let $\alpha_1\in\sr^{(1)}\cap S^{(1)}$ and $\sigma_2\in\sr^{(2)}$ be not orthogonal, then for $D\in\A(\alpha_1)$ the value of $c(D,\sigma_2)$ is not always uniquely determined, but there are some constraints: if $D\in\A(\alpha'_1)$ for some $\alpha'_1\in\sr^{(1)}\cap S^{(1)}$ different from $\alpha_1$ (we say that $D$ is not free) then $c(D,\sigma_2)=0$ or $\alpha'_1\not\perp\sigma_2$ (notice that the latter does not necessarily give rise to a plug). In particular, if both colors $D^+,D^-\in\A(\alpha_1)$ are not free, there can be no gluing with $\alpha_1$ not orthogonal to spherical roots of other weakly-color-connected components.
\end{remark}

Since we have classified all cuspidal weakly-color-connected spherical systems, it is now possible to construct any cuspidal spherical system by gluing.

\subsubsection*{Erasable and quasi-erasable weakly-color-connected components}

Let us introduce some further terminology: let $\S$ be a cuspidal spherical system. A weakly-color-connected component $\S'$ of $\S$ (or more generally a weakly-color-saturated localization $\S'$ of $\S$, i.e.\ a localization of $\S$ obtained by gluing some of its weakly-color-connected components) is called:
\begin{itemize}
\item isolated if $S'\perp(S\setminus S')$;
\item erasable if there exists a homogeneous $\ast$-distinguished subset of colors $\co_\ast$ of $\S'$ such that $c(D,\sigma)=0$, for all $D\in\co_\ast$ and all $\sigma\in\sr\setminus\sr'$;
\item quasi-erasable if there exists a non-empty $\ast$-distinguished subset of colors $\co_\ast$ of $\S'$ such that $c(D,\sigma)=0$, for all $D\in\co_\ast$ and all $\sigma\in\sr\setminus\sr'$.
\end{itemize}

It directly follows that if a cuspidal spherical system $\S$ admits two weakly-color-connected components such that one is isolated or both are quasi-erasable then $\S$ is decomposable. This obviously generalizes to weakly-color-saturated localizations of $\S$ on disjoint sets $S^{(1)},S^{(2)}\subset S$.

By abuse of terminology a weakly-color-connected component of rank 1 or 2 will be called tail if satisfies Definition \ref{def:tail} or \ref{def:exl} without the condition on the existence of the good distinguished subset of colors $\co'$. We will see below (Lemma~\ref{lem:A}) that in these hypotheses such $\co'$ always exists.

We now analyze all the cuspidal weakly-color-connected spherical systems which have been classified above. Case-by-case we start with a spherical system $\S^{(1)}$, see how it can be glued to another spherical system $\S^{(2)}$ such that $\S^{(1)}$ and $\S^{(2)}$ are weakly-color-saturated and say whether $\S^{(1)}$ is necessary isolated, erasable or quasi-erasable. 

{\it Rank 1.}
If a weakly-color-connected component is of rank 1 and equal to 
\[\begin{picture}(9000,600)(0,-300)\put(0,0){\usebox{\shortcm}}\end{picture}\]
then it is necessary isolated, the same is true if the spherical root of the weakly-color-connected component has support of type $\mathsf F_4$ or $\mathsf G_2$. If it is equal to 
\[\begin{picture}(9300,600)(-300,-300)\put(0,0){\usebox{\shortcsecondm}}\end{picture}\]
then it is necessarily erasable. If it is equal to the following case with a spherical root of type $b(m)$ $m\geq3$
\[\begin{picture}(7800,600)(-300,-300)\put(0,0){\usebox{\shortbsecondm}}\end{picture}\]
then it is necessarily quasi-erasable. If its spherical root is of type $b(m)$ $m\geq1$ ($\alpha_n\in S^p$ if $m\geq2$), $2b(m)$ $m\geq1$ or $d(m)$ $m\geq3$ then it is a tail. If its spherical root is of type $aa$ then it is a tail or it is glued to an $n$-comb component, if $n=1$ we get a rank 2 case (with support of type $\mathsf C_3$) which is isolated. If it is equal to
\[\begin{picture}(3900,600)(0,-300)\put(0,0){\usebox{\dynkinbthree}}\put(3600,0){\usebox{\gcircle}}\end{picture}\]
then it is isolated or it is glued to an $n$-comb component, if $n=1$ we get a rank 2 case which is isolated. If it is equal to the case of a spherical root of type $b(2)$ whose support does not meet $S^p$ then it is quasi-erasable or it is glued to two comb components so that the whole spherical system is decomposable. Notice that it remains the case of a spherical root of type $a(m)$ $m\geq1$, which will occur in some primitive cases below.

{\it Rank 2.} 
Analogously, if a weakly-color-connected component is of rank 2 then it is either isolated or erasable or quasi-erasable 
\[\begin{picture}(6000,600)(-300,-300)\put(0,0){\usebox{\dynkinbfour}}\multiput(0,0)(5400,0){2}{\usebox{\gcircle}}\end{picture}\]
or a tail or glued to an $n$-comb component (the corresponding cases are erasable, if $n=1$ and the 1-comb is primitive they are rank 3 spherical systems: \ppc-\ref{onecombaa}, \ppc-\ref{onecomb2a}, \ppc-\ref{onecombac}, \ppc-\ref{onecombef}, \ppc-\ref{onecombtwocombb}) or a $2$-comb or with two spherical roots of type $a(m)$ $m\geq2$ with overlapping supports (the two latter cases will be analyzed below).

{\it Clan \fR.} 
Let us start with $\a^{\faa}(p,p)$, $p>2$. It can only be glued to an $n$-comb component, if $n=1$ we get the case \ppc-\ref{onecombaa} which is isolated. Starting with $\a^\fo(p)$ the situation is very similar: we get the case \ppc-\ref{onecomb2a} which is isolated. Analogously with the case \fR-\ref{ac}: we get the cases \ppc-\ref{onecombac} which is isolated. The cases \fR-\ref{aapplusqplusp}, \fR-\ref{dcneven}, \fR-\ref{dcnodd} and \fR-\ref{efseven} are erasable. All the other cases of the clan \fR\ are isolated. 

{\it Clan \fS.} 
Let us consider the color-connected cases of the clan \fS. A weakly-color-connected component of this kind is isolated or erasable or it is equal to one of the following cases: 
\begin{itemize}
\item[-] $\bd^\fx(p,p)$ is quasi-erasable or, if $p= 2$, glued to a comb and, if the latter is primitive, as in \ppc-\ref{bxoneonetwoplusonecomb} which is isolated;
\item[-] $\bd^\fx(p-1,p)$ is quasi-erasable or, if $p=3$ glued to a comb and, if the latter is primitive, as in \ppc-\ref{bxthreetwoplusonecomb} which is isolated or, if $p=2$ namely $\a^\fx(1,1,1)$, not necessarily quasi-erasable;
\item[-] $\a^\fy(p,p)$ is erasable or glued to combs in a decomposable system or, if $p=2$, is not necessarily quasi-erasable;
\item[-] $\a^\fy(p-1,p)$ is quasi-erasable or glued to combs in a decomposable system or, if $p=2$, glued to a comb as in \ppc-\ref{ayonetwoplusonecombf} which is isolated;
\item[-] $\ab^\fy(p,p)$ and $\ad^\fy(p+1,p+1)$ are erasable or glued to combs and, if the latter are primitive, as in \ppc-\ref{byppplusonecomb}, \ppc-\ref{bytwotwoplusonecombc}, \ppc-\ref{dyplusonecomb} or \ppc-\ref{azplusonecomb} which are erasable;
\item[-] $\ab^\fy(p-1,p)$, $\ad^\fy(p-1,p+1)$ and $\a^\fz(1,3)$ are quasi-erasable;
\item[-] $\a^\fu(5)$ is quasi-erasable;
\item[-] $\ac^\fz(1,3)$ and $\b^\fw(3)$ can only be glued to combs in a decomposable system.
\end{itemize}

A weakly-color-connected component, if equal to a remaining case of the clan \fS, is isolated or erasable or is equal to \fS-\ref{axonepluspplusone}, \fS-\ref{axonepluspplusoneone}, \fS-\ref{aypplusqplusp} (of rank 5) or \fS-\ref{aypplusqpluspminusone} (of rank 4) which are quasi-erasable, or can only be glued to combs in a decomposable system.

{\it Clan \fT.}
Let us consider the cases of the clan \fT\ with only spherical roots of type $a(m)$. A weakly-color-connected component of this kind is isolated or erasable or it is equal to one of the following cases:
\begin{itemize}
\item[-] $\a^\fm(2p+1)$ is erasable or glued to combs in a decomposable system or, if $p=2$, is not necessarily quasi-erasable;
\item[-] $\a^\fm(2p+2)$ is quasi-erasable or glued to combs in a decomposable system;
\item[-] $\d^\fm(2p+1)$ and $\e^\fm(6)$ are quasi-erasable;
\item[-] \fT-\ref{acastnplusaone} and \fT-\ref{dsnplusaone} are quasi-erasable;
\item[-] \fT-\ref{dsastfour} is not necessarily quasi-erasable.
\end{itemize}

A weakly-color-connected component, if equal to a remaining case of the clan \fT, is isolated or erasable or is equal to $\b^\fm(2p+2)$, \fT-\ref{bcn} or \fT-\ref{bcn}', which are quasi-erasable, or can only be glued to combs in a decomposable system.

{\it Tails of type $c(m)$.}
For completeness we have to consider also weakly-color-connected spherical systems with a tail of type $c(m)$ (with overlapping support), but as components they are necessarily isolated, erasable or glued to combs in a decomposable system. This can be seen by collapsing the spherical root of type $c(m)$ (as in \ref{subsec:qpss} for the cases of the family \fc): one obtains only weakly-color-connected spherical systems that have already been considered above.

{\it Remaining cases.}
An $n$-comb that is a weakly-color-connected component by itself can be glued to a positive 1-comb (cases \ppc-\ref{onecombtwocombb} and \ppc-\ref{onecombthreecombd} which are erasable) or more generally to a spherical system with only spherical roots of type $a(1)$ and a free color (some cases listed in the family \fq\ of the clan \fS\ which are erasable).

Finally, if we have a rank 1 weakly-color-connected component with a spherical root $\sigma$ of type $a(m)$, denote by $D_1,D_2$ its colors, such that there exist two weakly-color-connected spherical roots $\sigma_1,\sigma_2\neq\sigma$ with $c(D_i,\sigma_i)\neq0$, $i=1,2$, then we obtain only the remaining cases listed in the family \fq\ of the clan \fS\ which are erasable.  

\subsubsection*{Type $\mathsf A$ connected subdiagrams of the Dynkin diagram}

\begin{lemma}\label{lem:A}
Let $S^{(1)}\sqcup S^{(2)}=S$ be such that the corresponding localizations $\S^{(1)},\S^{(2)}$ of $\S$ are weakly-color-saturated. Assume that the Dynkin diagram of $S^{(1)}$ has a connected component of type $\mathsf A_r$, let $\alpha_1,\ldots,\alpha_r$ be its simple roots (labelled as usual). Then
\begin{enumerate}
\item if $S^{(1)}\setminus\{\alpha_1,\alpha_r\}$ is orthogonal to $S^{(2)}$ then $\alpha_1+\ldots+\alpha_r\in\sr$ (and the localization of $\S$ on $\{\alpha_1,\ldots,\alpha_r\}$ is a weakly-color-connected component by itself) or $\S^{(1)}$ is quasi-erasable; 
\item if $S^{(1)}\setminus\{\alpha_1\}$ is orthogonal to $S^{(2)}$ then $\S^{(1)}$ is erasable.
\end{enumerate}
\end{lemma}

\begin{proof}\
\begin{enumerate}
\item It is enough to look at the weakly-color-connected spherical systems of rank $>1$ that are not necessarily quasi-erasable: notice that they are quasi-erasable if their support meets a connected component of type $\mathsf A$ of the Dynkin diagram of $S$.
\item Similarly, it is enough to look at the weakly-color-connected spherical systems that are not necessarily erasable: notice that they are erasable if their support contains an extremal simple root of a connected component of type $\mathsf A$ of the Dynkin diagram of $S$. 
\end{enumerate}
\end{proof}

The following completes the proof of Theorems \ref{thm:pss} and \ref{thm:ppc}.

\begin{proposition}
Let $\S$ be a cuspidal spherical $R$-system without tails and positive $n$-combs with $n>1$. If $\S$ is not primitive and has no primitive positive 1-combs, there exists a decomposition $S^{(1)}\sqcup S^{(2)}=S$ such that the corresponding localizations $\S^{(1)},\S^{(2)}$ of $\S$ are weakly-color-saturated and both quasi-erasable, therefore $\S$ is decomposable.
\end{proposition}

\begin{proof}
We are left to consider the cuspidal spherical $R$-systems $\S$, without tails and positive $n$-combs with $n>1$, with a weakly-color-connected component that is not quasi-erasable. A case-by-case analysis of the weakly-color-connected spherical systems that are not necessarily quasi-erasable, together with Lemma~\ref{lem:A}, shows that, if $\S$ is not primitive and has no primitive positive 1-combs, we can always choose such a decomposition $S^{(1)}\sqcup S^{(2)}=S$.
\end{proof}

Furthermore, another direct consequence of Lemma~\ref{lem:A} is the following, which applies more generally to all spherical systems.

\begin{proposition}\label{prop:era}
Let $\S$ be a spherical $R$-system. Then there exists a decomposition $S^{(1)}\sqcup S^{(2)}=S$ such that the corresponding localization $\S^{(1)}$ is weakly-color-saturated and erasable.
\end{proposition}

\begin{proof}
Notice that for every decomposition $S^{(1)}\sqcup S^{(2)}=S$ either $S^{(1)}\perp S^{(2)}$ or the Dynkin diagram of $S^{(1)}$ (or of $S^{(2)}$) has a connected component of type $\mathsf A$. Therefore, by Lemma~\ref{lem:A}, for all spherical $R$-systems $\S$ there exists such a decomposition.
\end{proof}

Therefore, by the way, we have also proved the following

\begin{corollary}
The whole set of colors of a spherical system is distinguished and good.
\end{corollary}

\begin{proof}
Let $S^{(1)}\sqcup S^{(2)}=S$ be a decomposition as in Proposition~\ref{prop:era}. Call $\co_\ast^{(1)}$ a corresponding homogeneous distinguished subset of colors of $\S^{(1)}$, it can be identified with a good distinguished subset of colors of $\S$. Proposition~\ref{prop:era} can then be applied recursively to the quotient $\S/\co_\ast^{(1)}$ (which is equal to $\S^{(2)}$ up to induction).
\end{proof}

\section{Quotients of general spherical systems}\label{sec:prop}

\refstepcounter{subsection}

\begin{theorem}\label{prop:dist}
All distinguished subsets of colors of a spherical system are good.
\end{theorem}

Here we use the notion of weak-color-connectedness defined in \ref{subsec:col}.

First, notice that it is enough to prove the theorem for minimal distinguished subsets $\co_\ast$ such that $\sr\cap\sr/\co_\ast=\emptyset$. Indeed, let $\sigma\in\sr\cap\sr/\co_\ast$, then for all $D\in\co_\ast$ one has $c(D,\sigma)=0$: for all $\alpha\in S$, $D\in\co_\ast\cap\co(\alpha)$ implies that $\alpha\perp\mathrm{supp}\,\sigma$ or $D\in\A(\alpha)$ with $\A(\alpha)\setminus\co_\ast\neq\emptyset$.  

Moreover, recall that all minimal distinguished subsets of rank $\leq2$ spherical systems or of primitive spherical systems or of spherical systems with a primitive positive 1-comb are good (Remarks \ref{rem:lgood} and \ref{rem:pgood}). The same is true for all weakly-color-connected spherical systems, indeed, recall that systems with a tail of type $c(m)$ (with overlapping support) behave as their analogues obtained by collapsing the tail.  

Other spherical systems with a minimal distinguished subset of colors $\co_\ast$ with $\sr\cap\sr/\co_\ast=\emptyset$ are the following:
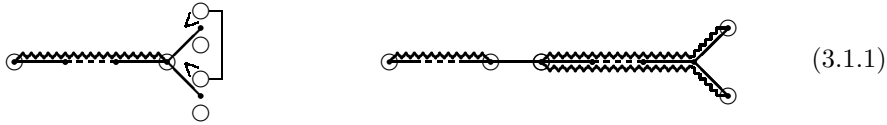
\begin{equation}\label{eq:further}
\parbox{10.6cm}{
\[\begin{picture}(7650,4200)(-300,-2100)\put(0,0){\usebox{\mediumam}}\put(5400,0){\usebox{\bifurc}}\multiput(6600,-1200)(0,2400){2}{\usebox{\aone}}\multiput(6900,-600)(0,2400){2}{\line(1,0){450}}\put(7350,-600){\line(0,1){2400}}\put(6600,-600){\usebox{\tonw}}\put(6600,1800){\usebox{\tosw}}\end{picture}
\qquad\qquad\qquad
\begin{picture}(12600,4200)(-300,-2100)\put(0,0){\usebox{\shortam}}\put(3600,0){\usebox{\edge}}\put(5100,-1500){\diagramdsn}\end{picture}\]
}
\end{equation}

We claim that no other spherical system admits a minimal distinguished subset of colors $\co_\ast$ with $\sr\cap\sr/\co_\ast=\emptyset$.
 
Let $\S=(S^p,\sr,\A)$ be a spherical $R$-system with set of colors $\co$ and a minimal distinguished subset $\co_\ast\subset\co$. Let $\S^{(1)},\S^{(2)}$ be weakly-color-saturated localizations of $\S$ corresponding to $S^{(1)}\sqcup S^{(2)}=S$. If $\co_\ast\not\subset\co^{(1)}$, then every non-empty subset of $\co_\ast\cap\co^{(1)}$ that is distinguished in $\S^{(1)}$ must contain a color $D$ such that $c(D,\sigma)<0$ for some $\sigma\in\sr^{(2)}$, and vice versa. We will say that $D\in\co_\ast\cap\co^{(1)}$ and $\sigma\in\sr^{(2)}$ such that $c(D,\sigma)<0$ are in {\it gluing relation}. Assuming that $\co_\ast\cap\co^{(2)}$ is good distinguished in $\S^{(2)}$, if $\sigma\in\sr^{(2)}$ is in gluing relation with some $D\in\co_\ast\cap\co^{(1)}$, then there exists an element of $\mathbb N(\co_\ast\cap\co^{(2)})$ that is positive on $\sr^{(2)}$ and strictly positive on $\sigma$.

\begin{lemma}\label{lem:sg}
Let $\S$ be a not weakly-color-connected spherical $R$-system with a minimal distinguished subset of colors $\co_\ast$. Then there exist non-empty subsets $S^{(1)}$, $S^{(2)}$ with $S^{(1)}\sqcup S^{(2)}=S$ and weakly-color-saturated $\S^{(1)},\S^{(2)}$ such that $\co_\ast\subset\co^{(1)}$ or satisfying the following condition. 
\begin{equation}\label{eq:sg}
\left.
\parbox{65ex}{
There exist 
\begin{itemize}
\item[-] $D^{(1)}\in\co_\ast\cap\co^{(1)}$ and $\sigma^{(1)}\in\sr^{(1)}$ with $c(D^{(1)},\sigma^{(1)})\geq0$, 
\item[-] $D^{(2)}\in\co^{(2)}$ and $\sigma_k^{(2)}\in\sr^{(2)}$, for $k=1,\ldots,n$ and $n\geq 1$ (actually $\leq3$), with $c(D^{(2)},\sigma_k^{(2)})\geq0$, 
\end{itemize}
such that 
\begin{quote}
$D\in\co_\ast\cap\co^{(i)}$ and $\sigma\in\sr^{(j)}$ with $i\neq j$ are in gluing relation only if $D=D^{(1)}$ and $\sigma=\sigma_k^{(2)}$, or $D=D^{(2)}$ and $\sigma=\sigma^{(1)}$.
\end{quote}
}
\right\}
\end{equation}
\end{lemma}

\begin{proof}
First, we can assume that all weakly-color-connected components satisfying 
\begin{equation}\label{eq:i}
\parbox{65ex}{
at least two connected components of the support are non-orthogonal to their complement
}
\end{equation}
have a non-free color that is in gluing relation with some spherical root of other weakly-color-connected components. Indeed, it can be checked case-by-case that a cuspidal weakly-color-connected spherical system that can satisfy the condition \eqref{eq:i} has no distinguished subset $\co_\ast$ such that 
\begin{itemize}
\item every minimal distinguished subset of $\co_\ast$ contains a free color and 
\item there exists a positive element of $\mathbb N\co_\ast$ that is strictly positive on some spherical root.
\end{itemize}

A non-free color $D$ can be in gluing relation with only one spherical root $\sigma$ of other weakly-color-connected components. If there is such a color $D$, cut the Dynkin diagram between the support of $\sigma$ and $\alpha$ for all $\alpha\in S$ such that $D\in\co(\alpha)$. Define $S^{(1)}$ as the minimal subset of $S$ containing the connected component of the Dynkin diagram (after the cut prescribed above) that contains $\mathrm{supp}\,\sigma$ and such that $\S^{(1)}$ is weakly-color-saturated. By the above assumption, $S^{(1)}$ and its complement do the job.

We are left with the case when there is no weakly-color-connected component satisfying the condition \eqref{eq:i}. Here there always exists a weakly-color-connected component whose support $S'$ contains only one simple root that is non-orthogonal to $S\setminus S'$: it is thus enough to set $S^{(1)}=S'$ or $S^{(1)}=S\setminus S'$.
\end{proof}

Under the condition \eqref{eq:sg} we conclude the proof. We keep the notation of the lemma.

Assume $n=1$ and set $\sigma^{(2)}=\sigma_1^{(2)}$. We claim that 
\begin{itemize}
\item[(a)] $D^{(1)}$ and $D^{(2)}$ are positive on $\sr^{(1)}$ and $\sr^{(2)}$, respectively, or 
\item[(b)] there exist $i$ (equal to 1 or 2) and a spherical root $\sigma\neq\sigma^{(i)}$ in $\sr^{(i)}$ with $c(D^{(i)},\sigma)>0$ (if $c(D^{(i)},\sigma^{(i)})>0$ this means that $D^{(i)}$ is not free).
\end{itemize}

Indeed, by the list of not weakly-color-connected rank 2 spherical systems it follows that $c(D^{(1)}+D^{(2)},\sigma^{(i)})\leq0$. Let $\tilde D=\sum n_D D$ be a positive element of $\mathbb N_{>0}\co_\ast$. Set $\tilde D^{(i)}=\sum n_D D$ for all $D\in\co_\ast\cap\co^{(i)}\setminus\{D^{(i)}\}$, notice that $\co_\ast\cap\co^{(i)}=\{D^{(i)}\}$ if and only if $D^{(i)}$ is positive on $\sr^{(i)}$. If there exists no $\sigma\neq\sigma^{(i)}$ in $\sr^{(i)}$ with $c(D^{(i)},\sigma)>0$, then $\tilde D^{(i)}$ is positive on $\sr^{(i)}\setminus\{\sigma^{(i)}\}$. If for both $i=1,2$ there exists no $\sigma\neq\sigma^{(i)}$ with $c(D^{(i)},\sigma)>0$, then $c(\tilde D^{(i)},\sigma^{(i)})\geq0$ for $i=1,2$. 
Therefore, under the latter condition, if (for $i$ equal to 1 or 2) $D^{(i)}$ is not positive on $\sr^{(i)}$ then $\co_\ast\cap\co^{(i)}\setminus\{D^{(i)}\}$ is not empty and $\tilde D^{(i)}$ is positive on $\sr^{(i)}$: a contradiction. 

To conclude. 
\begin{itemize}
\item[(a)] If the two colors $D^{(1)}$ and $D^{(2)}$ are positive on $\sr^{(1)}$ and $\sr^{(2)}$, respectively, namely $\co_\ast=\{D^{(1)},D^{(2)}\}$ and $c(D^{(1)}+D^{(2)},\sigma^{(i)})=0$, it is enough to look at the rank 2 case with spherical roots $\sigma^{(1)}$ and $\sigma^{(2)}$. 
\item[(b)] If $D^{(2)}$ was not free then $\sigma^{(2)}$ would be of type $a(m)$, but in this case $\sigma^{(1)}$ and $\sigma^{(2)}$ would be weakly-color-connected (contradiction).
\end{itemize}

Finally, if $n>1$ the argument is analogous: the only spherical systems (not primitive and without primitive positive 1-combs) that actually occur under the given hypothesis are reported in \eqref{eq:further} above.

\appendix

\section{Spherical systems of rank $\leq2$}\label{app:low}

Let $S_i$, $i=1,2$, be subsets of $S$ with $S_1\sqcup S_2=S$ and $S_1\perp S_2$. Let $\S_i=(S^p_i,\sr_i,\A_i)$, $i=1,2$, be spherical $R_i$-systems. The direct product of $\S_1$ and $\S_2$ is the spherical $R$-system $(S_1^p\sqcup S_2^p,\sr_1\sqcup\sr_2,\A_1\sqcup\A_2)$ where $c(D,\sigma)=0$ for all $D\in\A_i$, $\sigma\in\sr_j$ and $i\neq j$.

Here we give the Luna diagrams of all cuspidal spherical systems of rank 1 and of all cuspidal spherical systems of rank 2 that are not direct product of two spherical systems of rank 1, together with their minimal quotient spherical systems. 

For cuspidal spherical $R$-systems of rank 1, $n$ denotes the rank of $R$.

For cuspidal spherical systems of rank 2 that are not direct product of two spherical systems of rank 1, $p$ and $q$ denote the rank of the supports of the two spherical roots: we use identifications between spherical roots for small $p$ and $q$ as stated in Definition~\ref{def:sr}.

\subsection*{Rank 1}

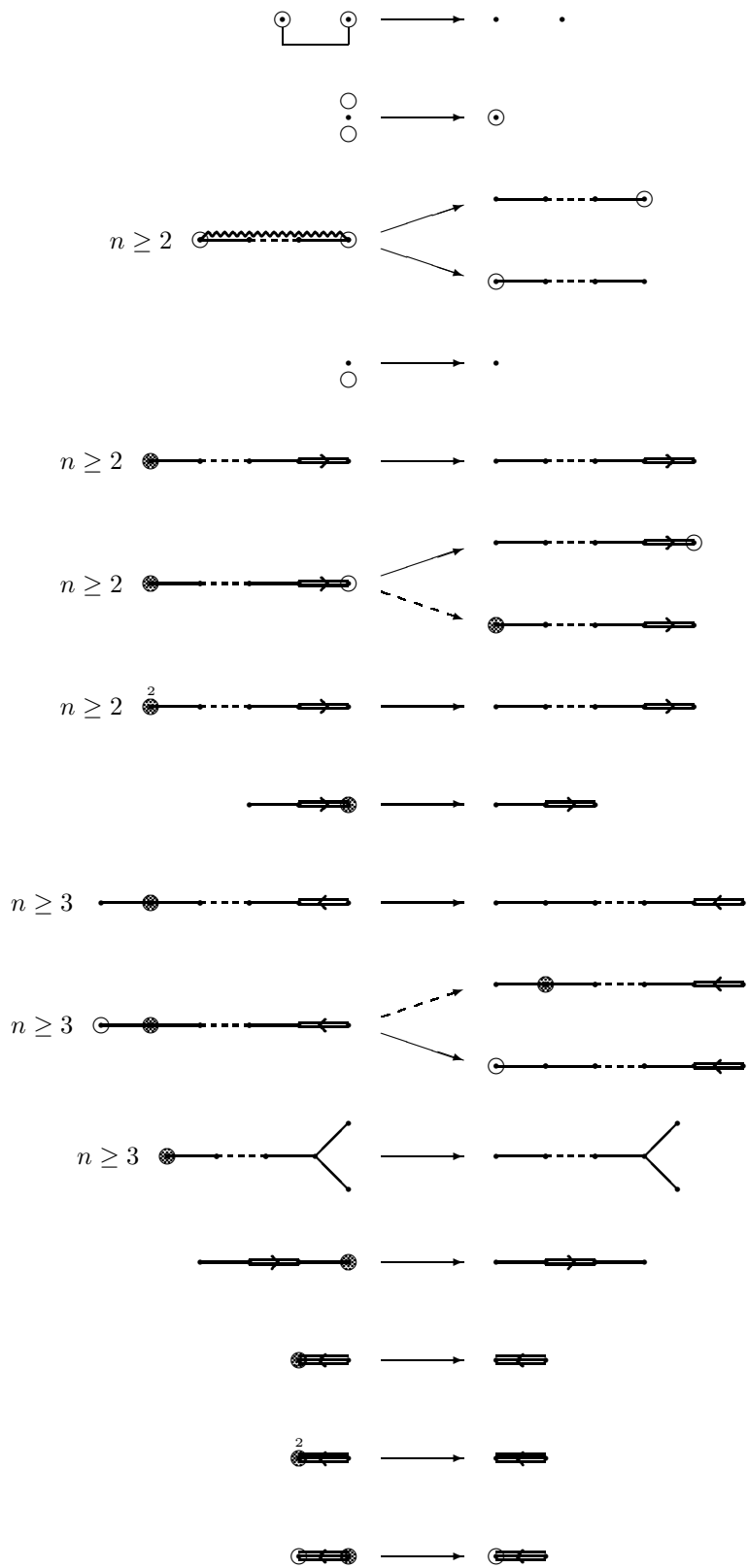
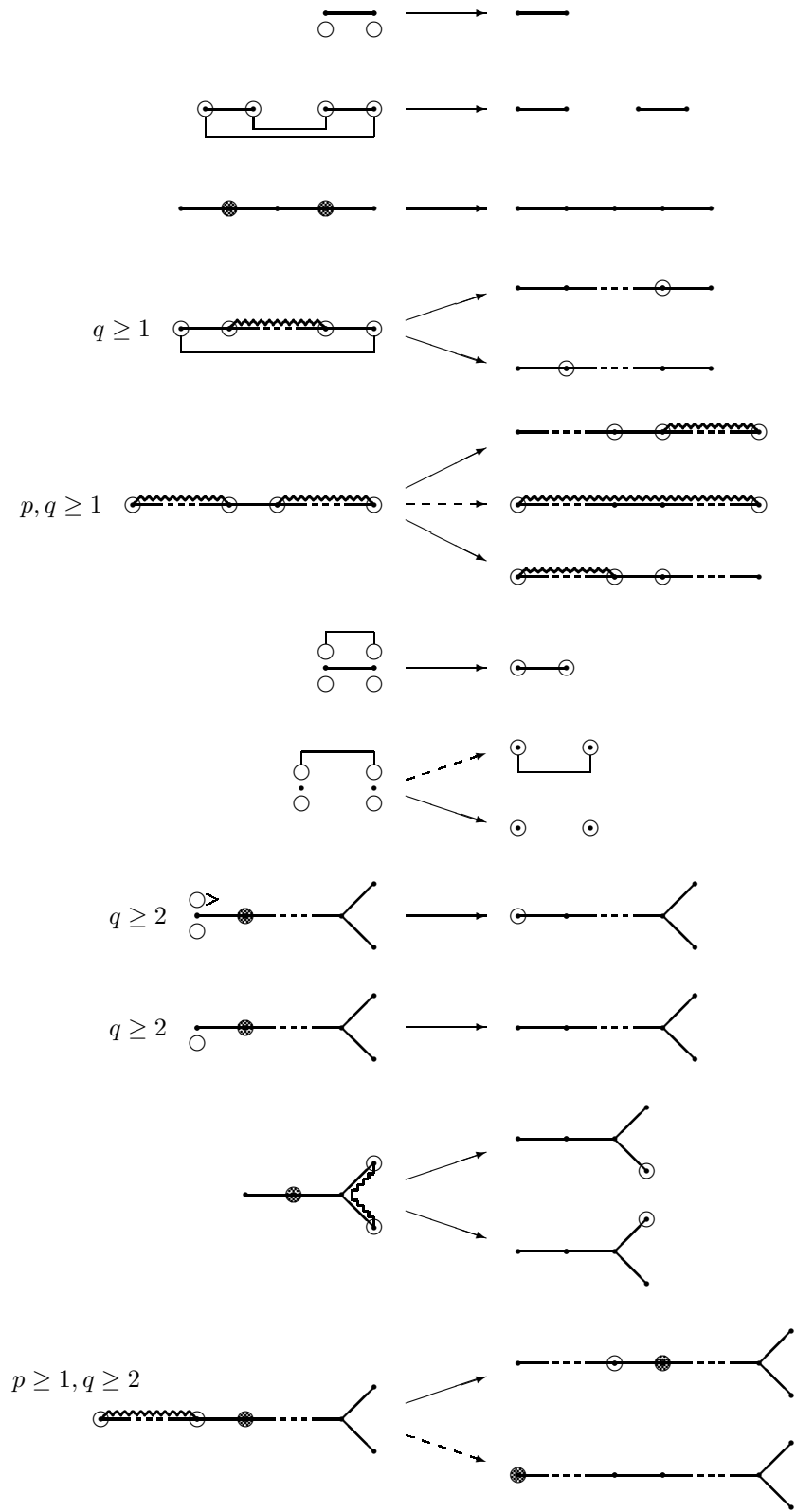
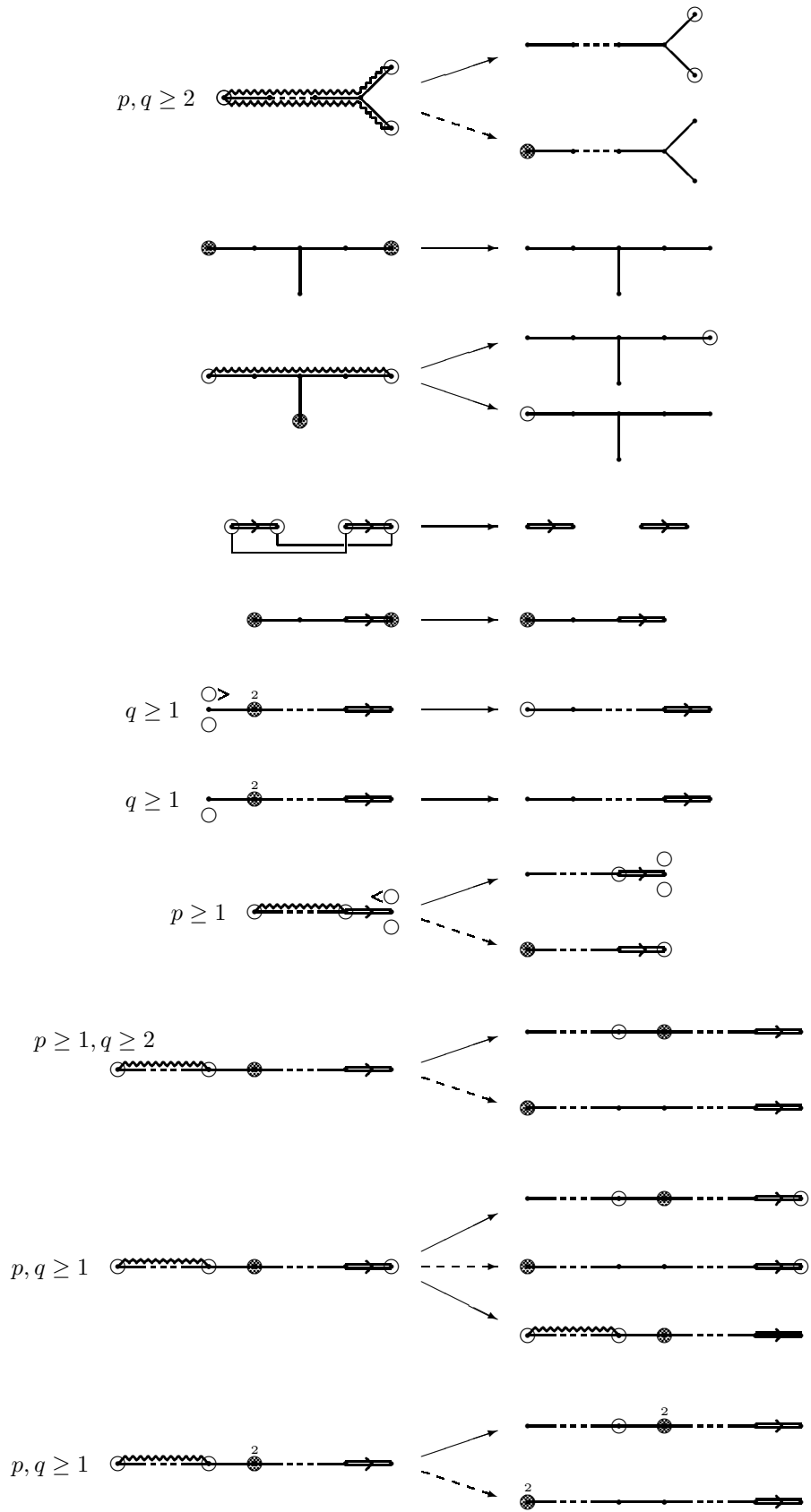
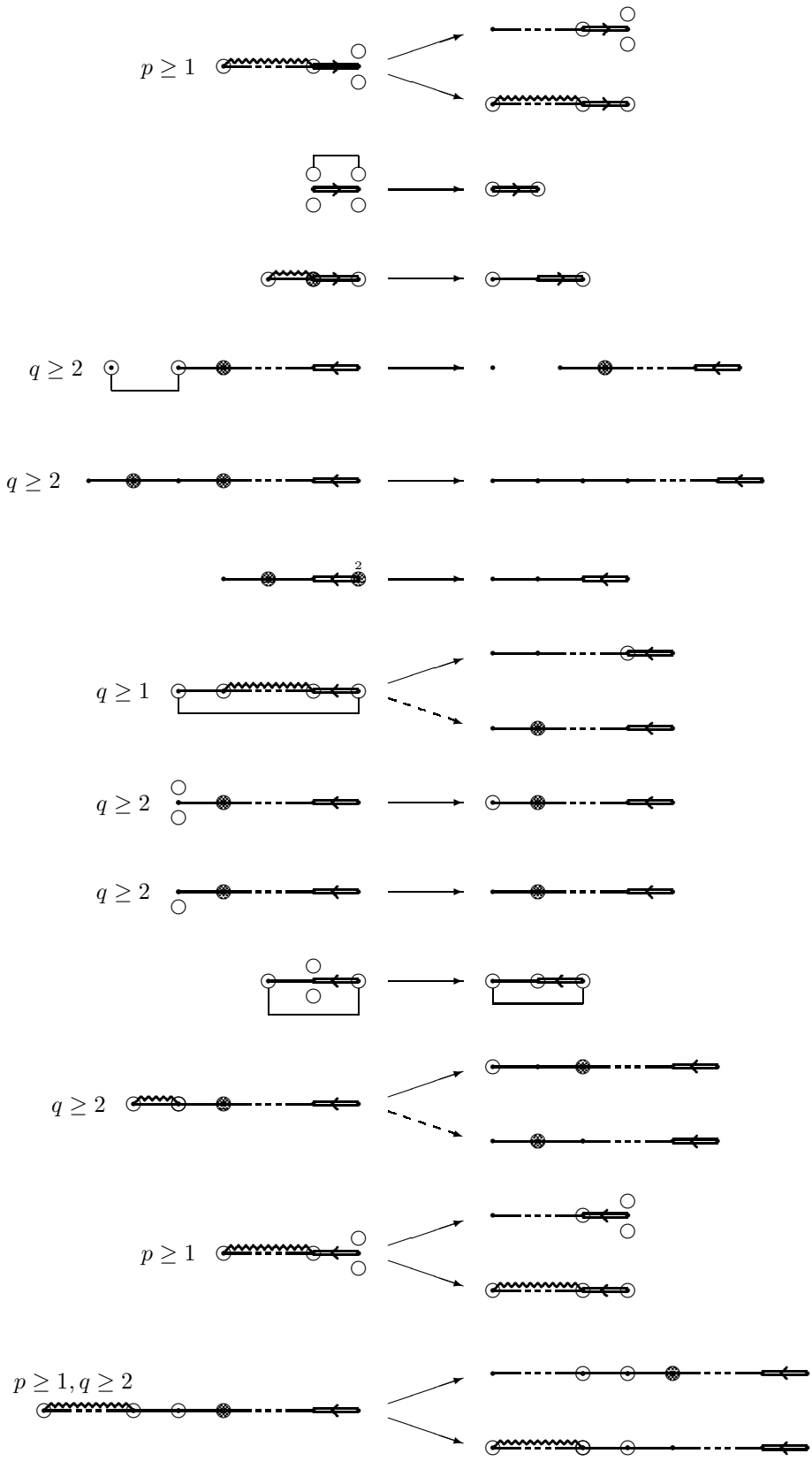
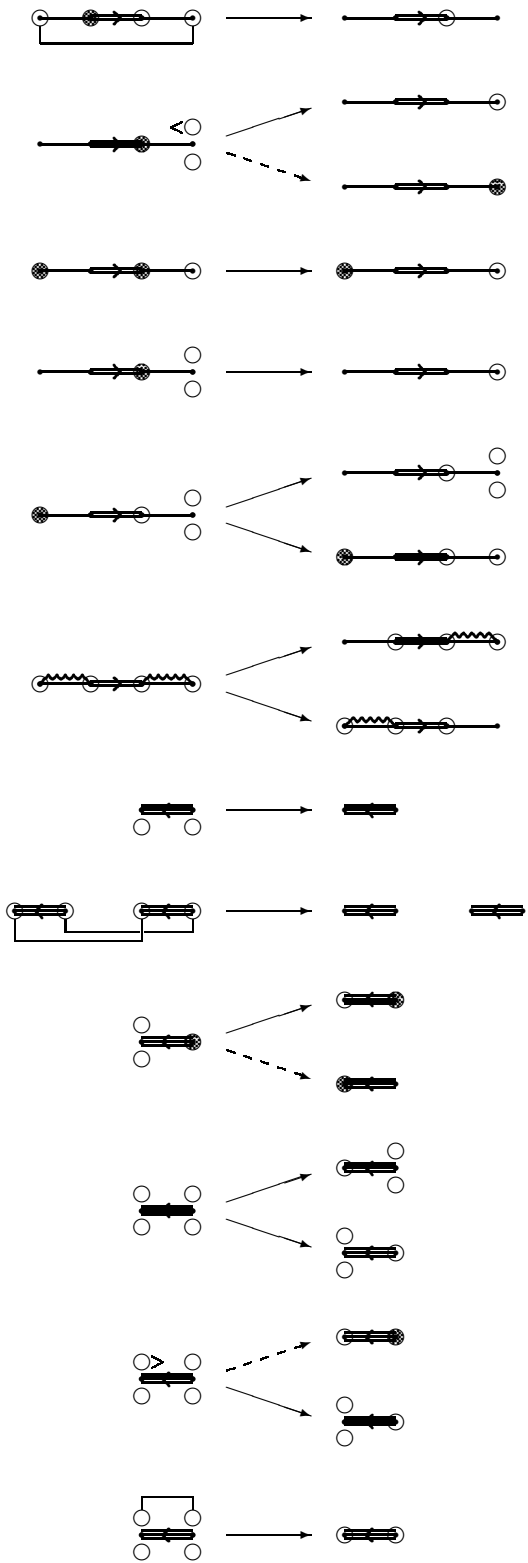
\[
\]

\subsection*{Rank 2}

\[%
\]

\section{Minimal quotients}\label{app:qpss}

Here we give all the minimal quotients of primitive spherical systems and of spherical systems with a primitive positive 1-comb (of rank $>2$), up to external automorphism. 

We omit the quotients of the p.s.s.'s of the clan \fR, which are all homogeneous. They are described in \ref{subsec:qpss}.

\subsection*{Clan \fS}

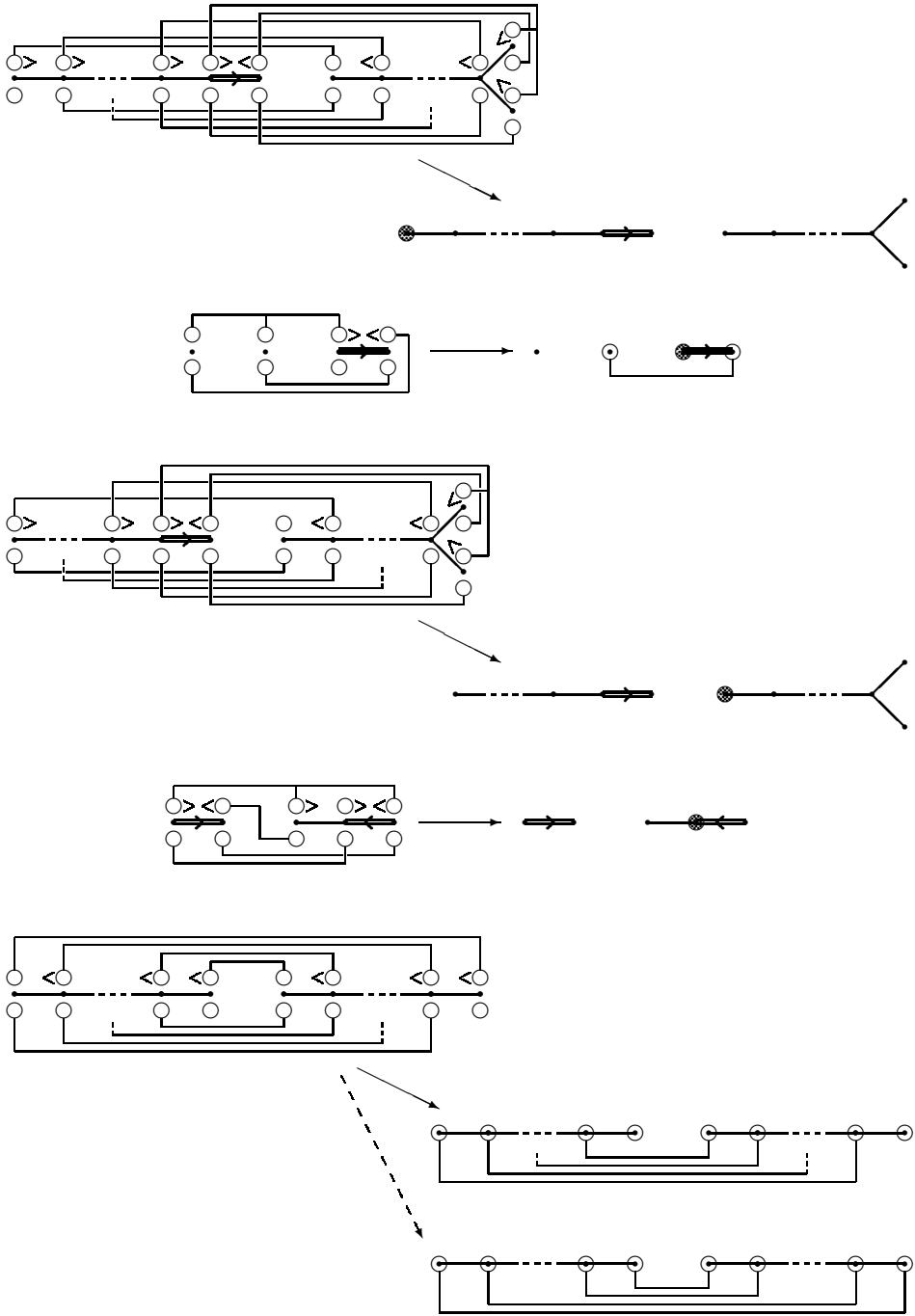
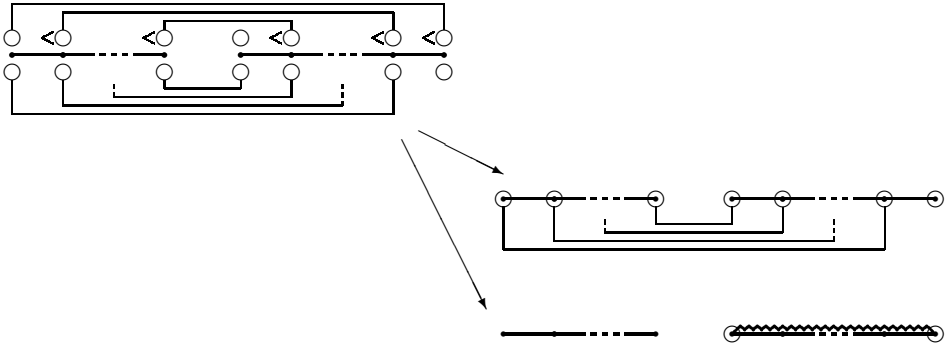
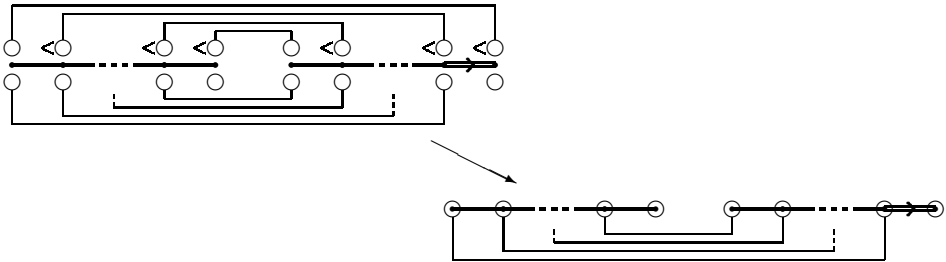
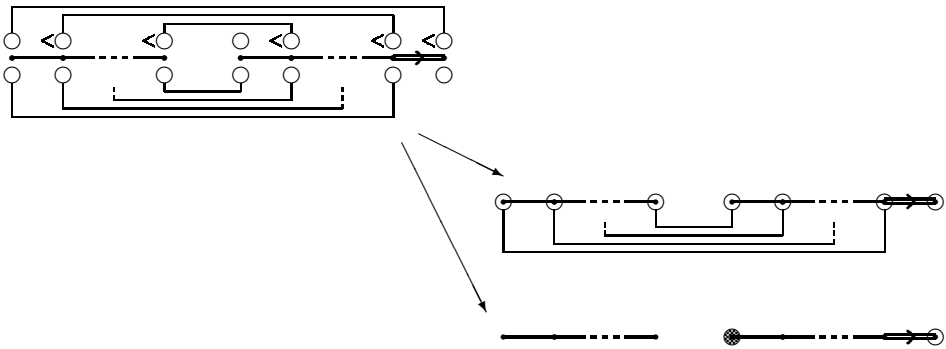
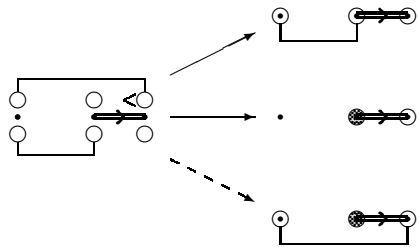
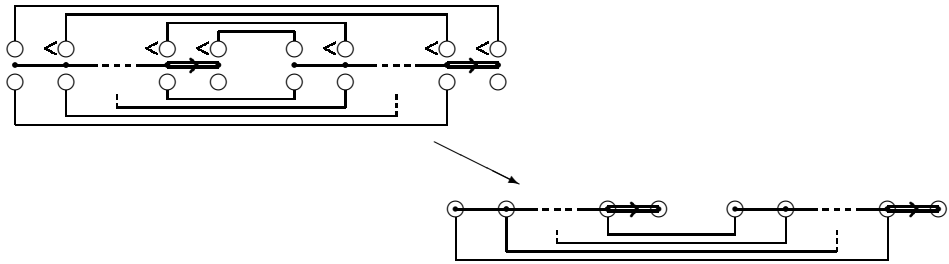
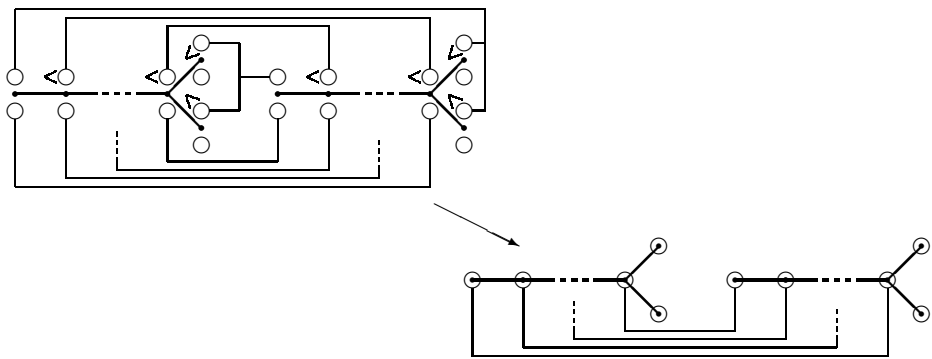
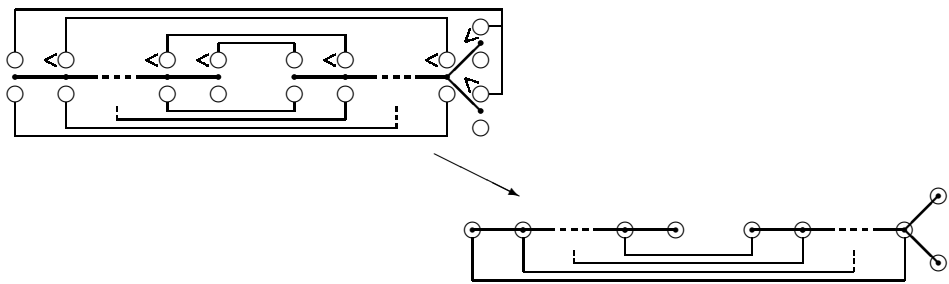
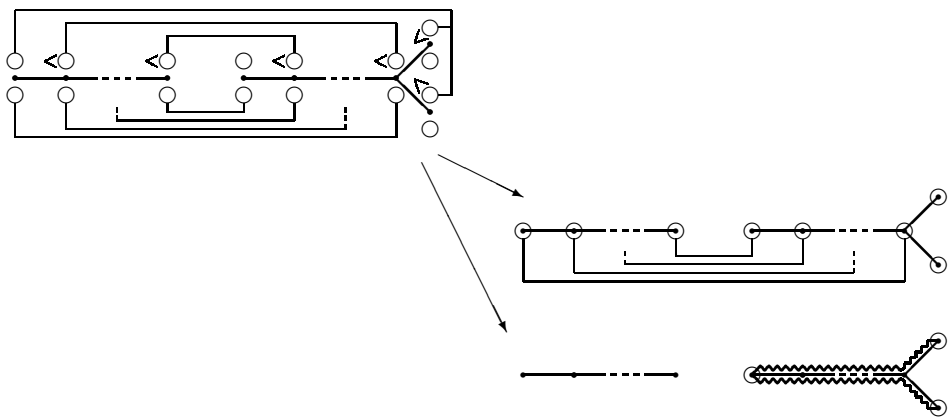
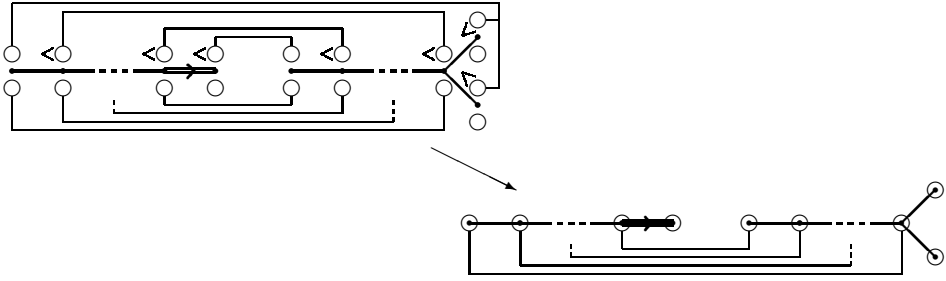
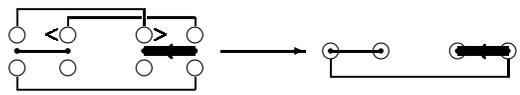
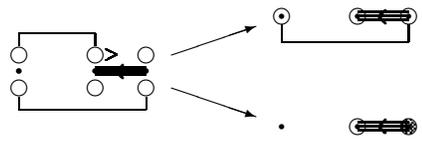
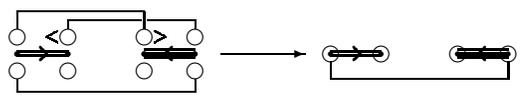
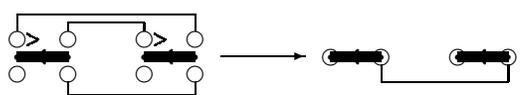
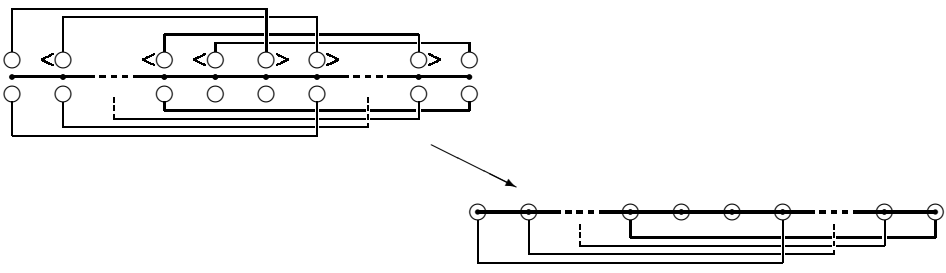
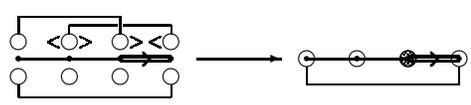
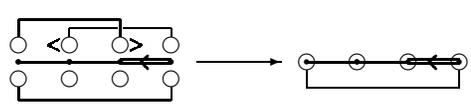
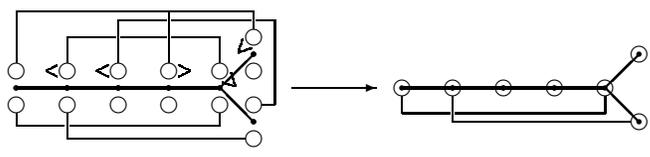
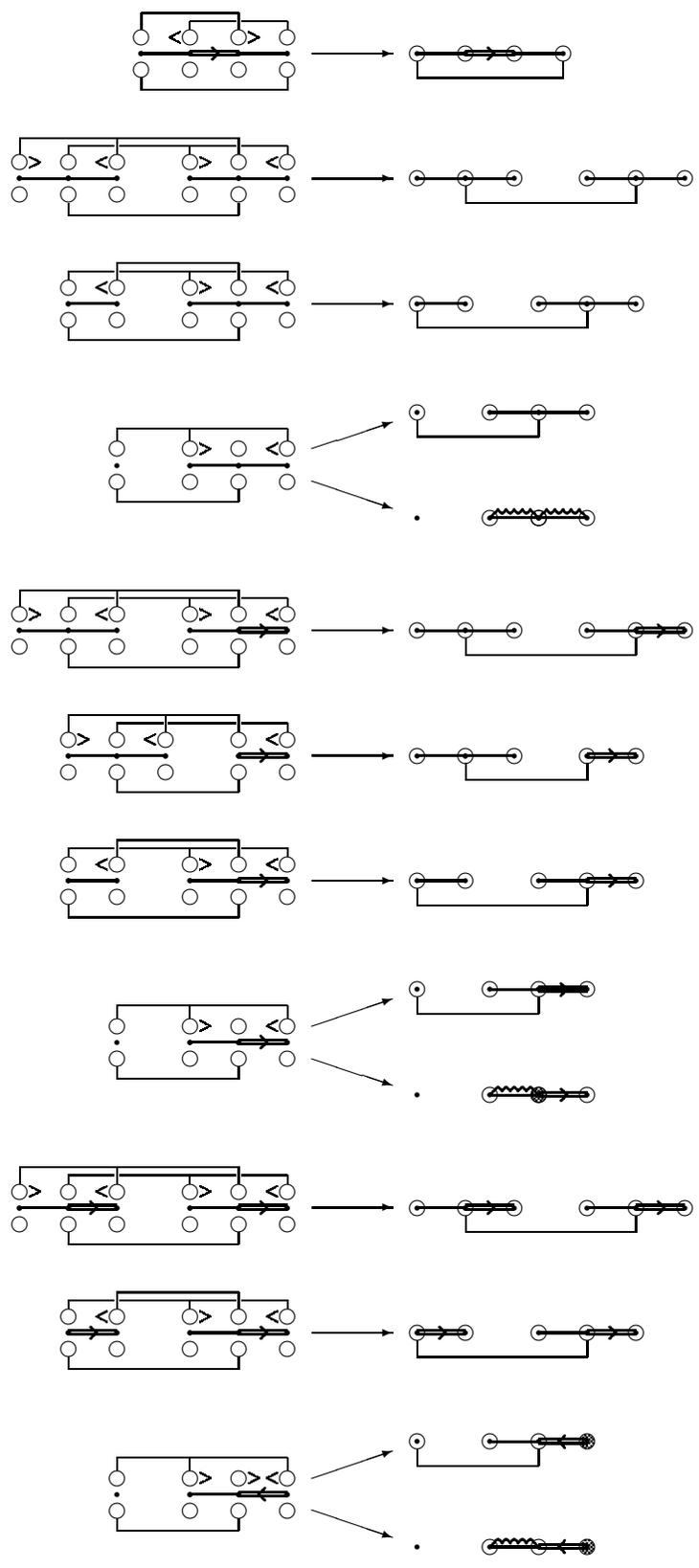
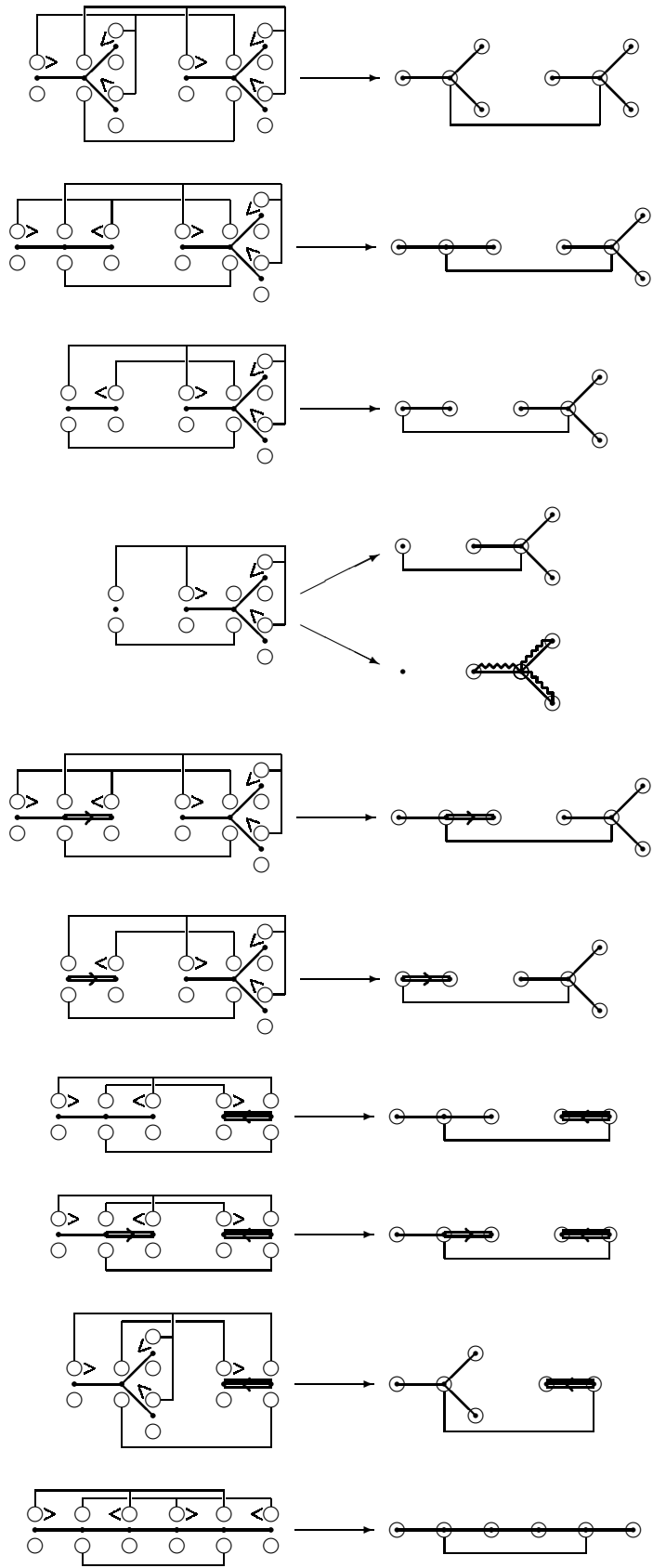
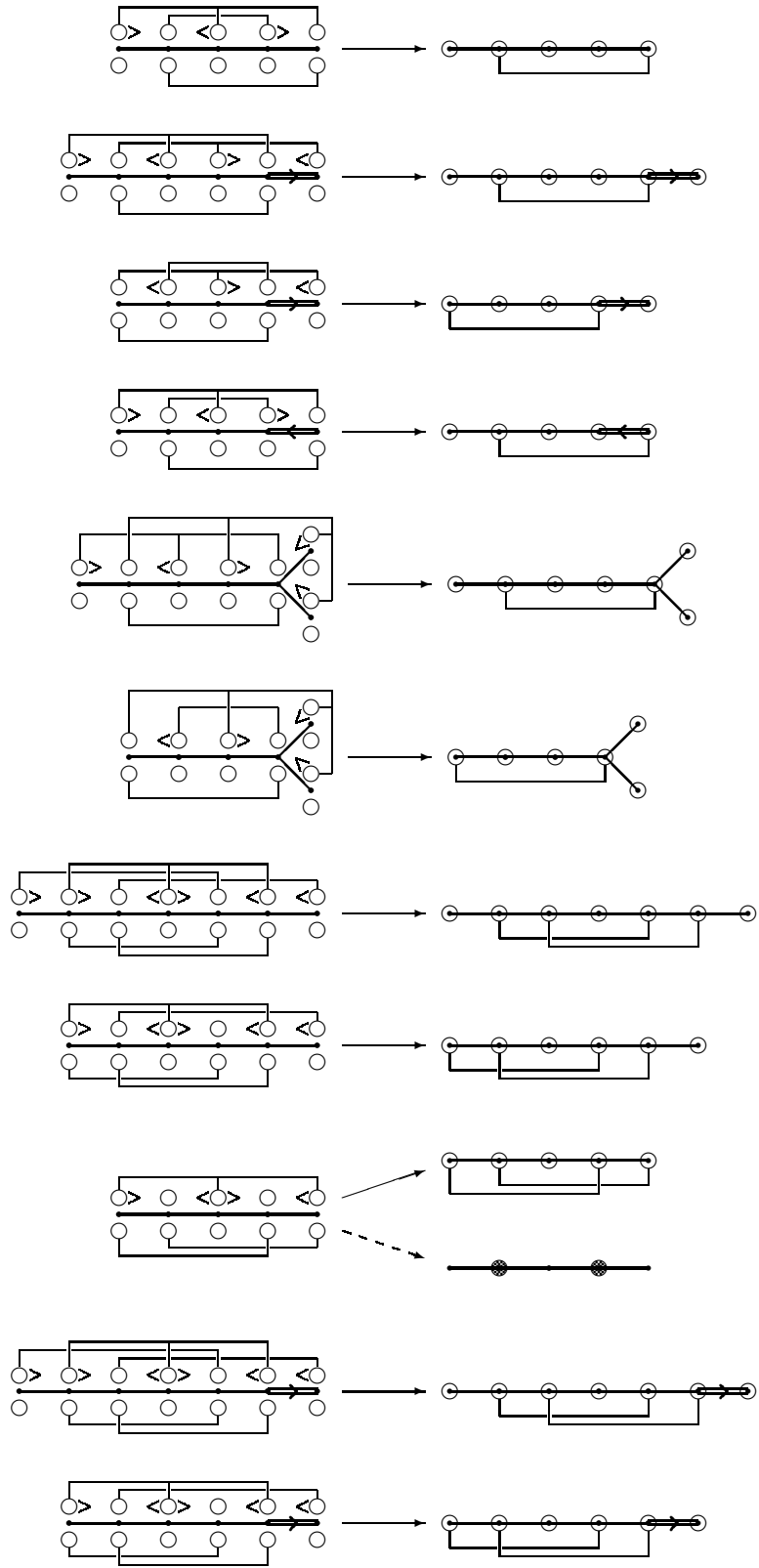
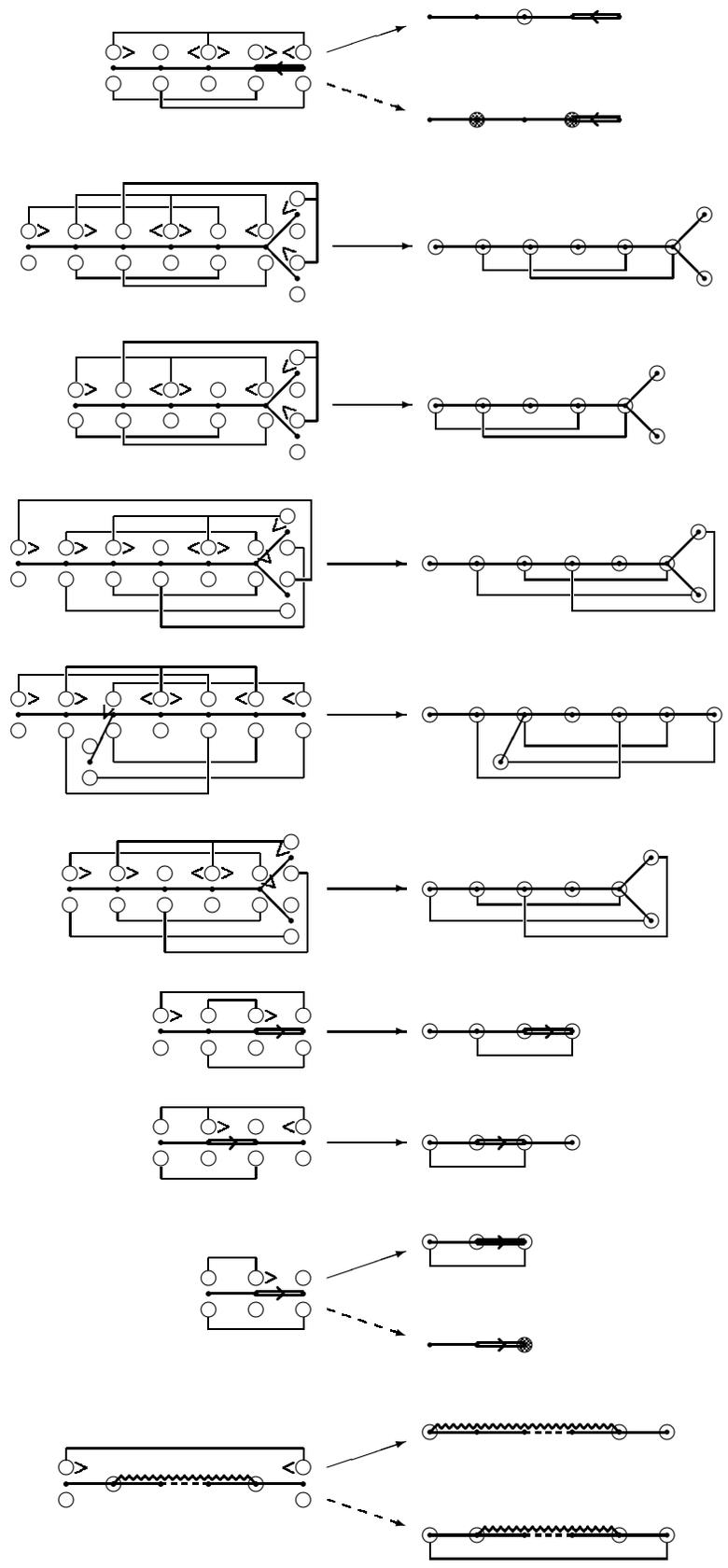
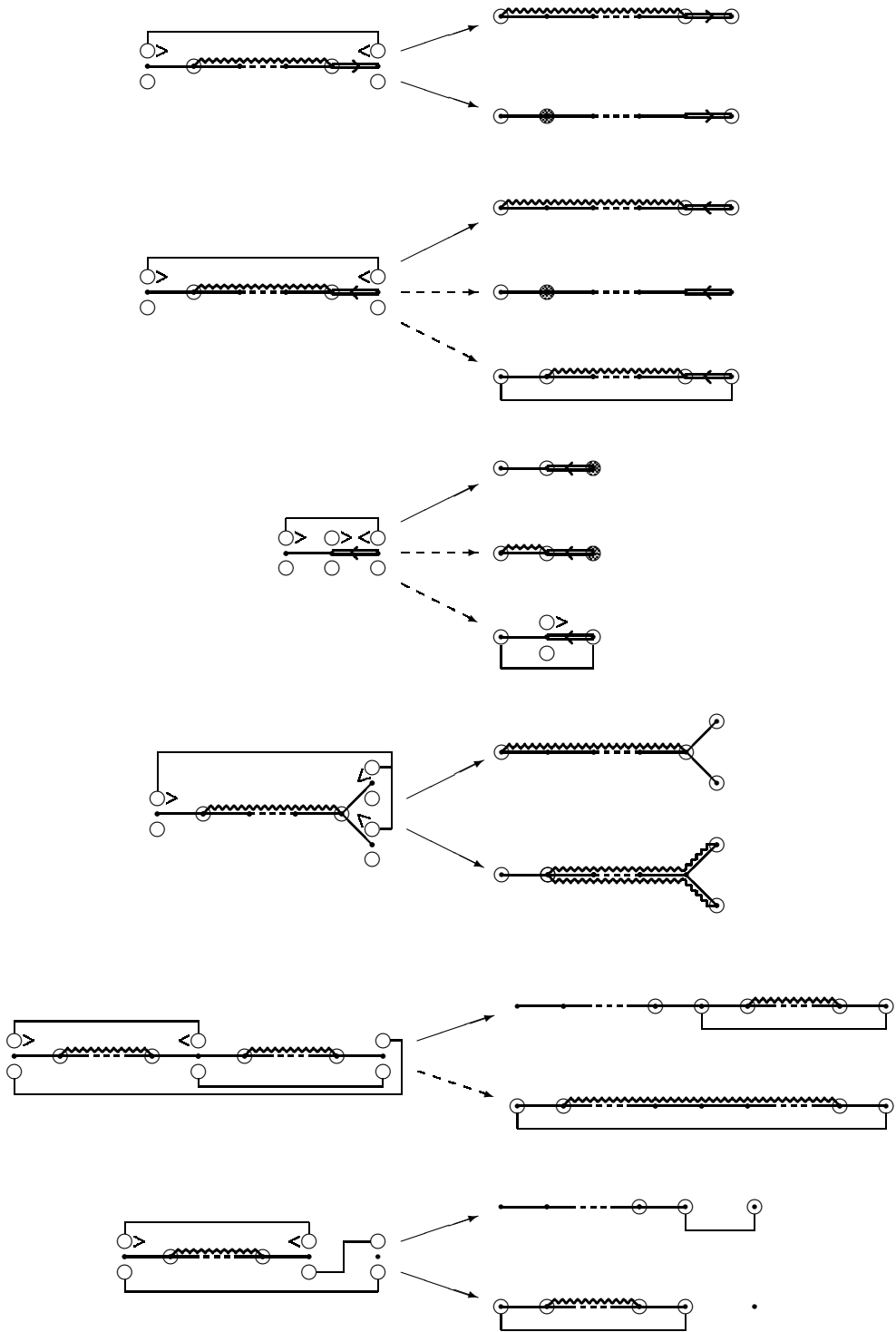
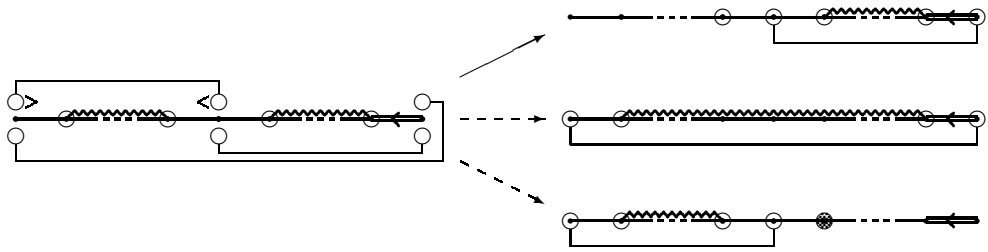
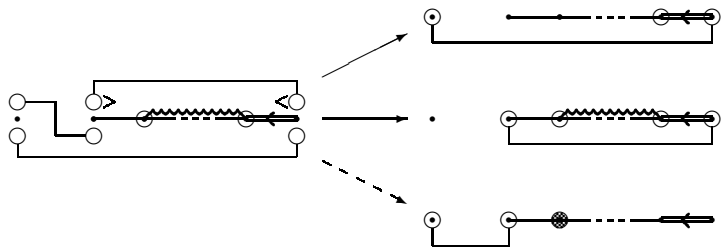
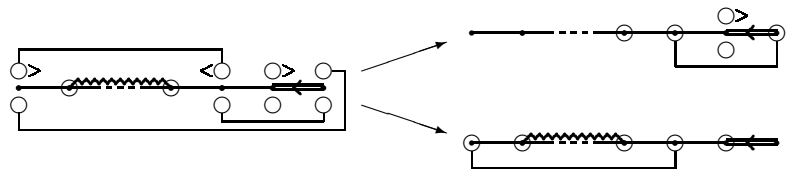
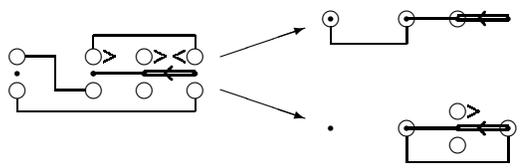
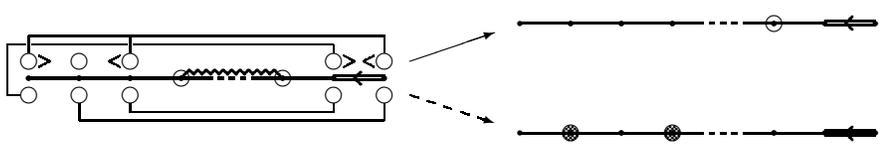
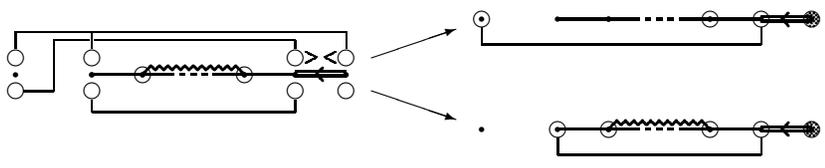
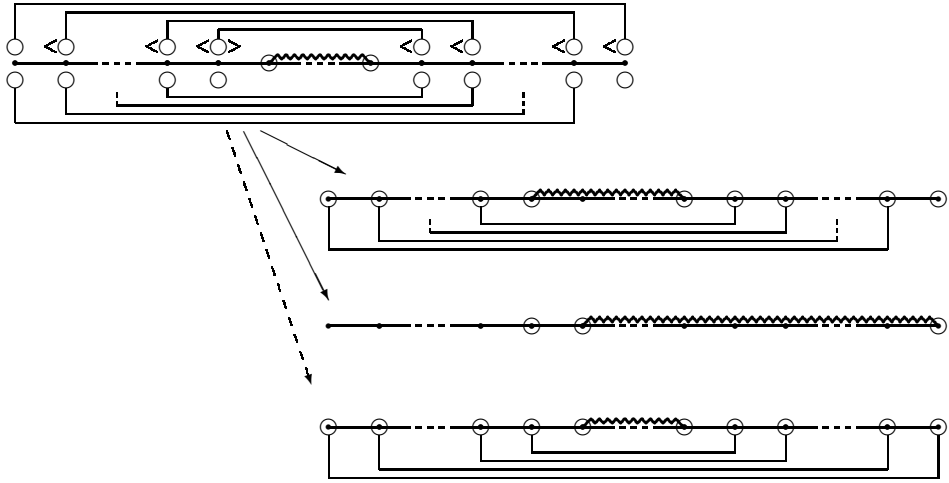
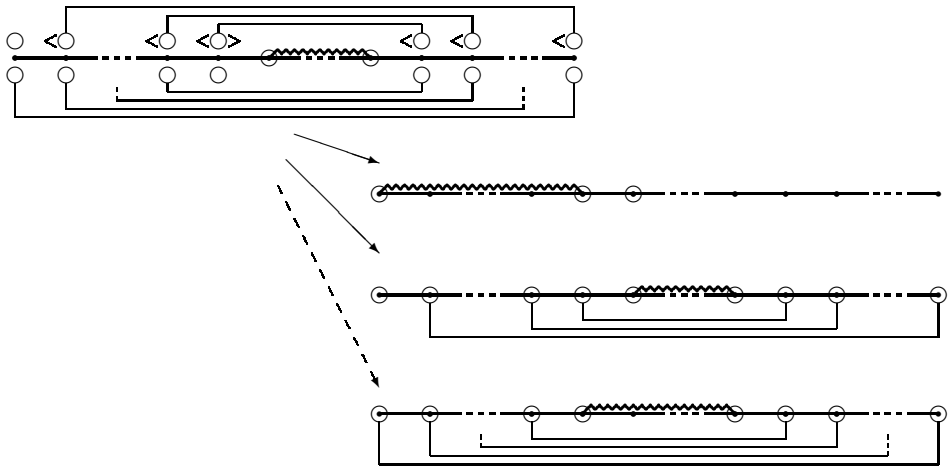
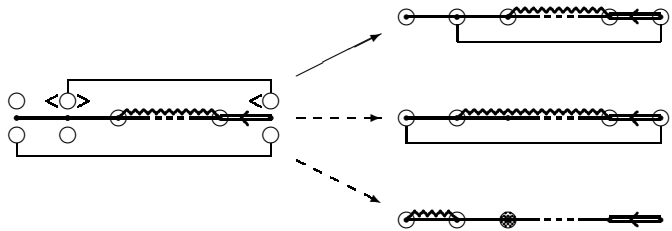
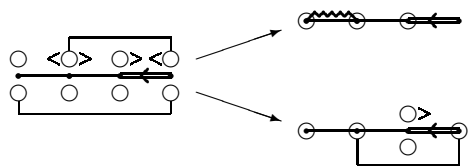
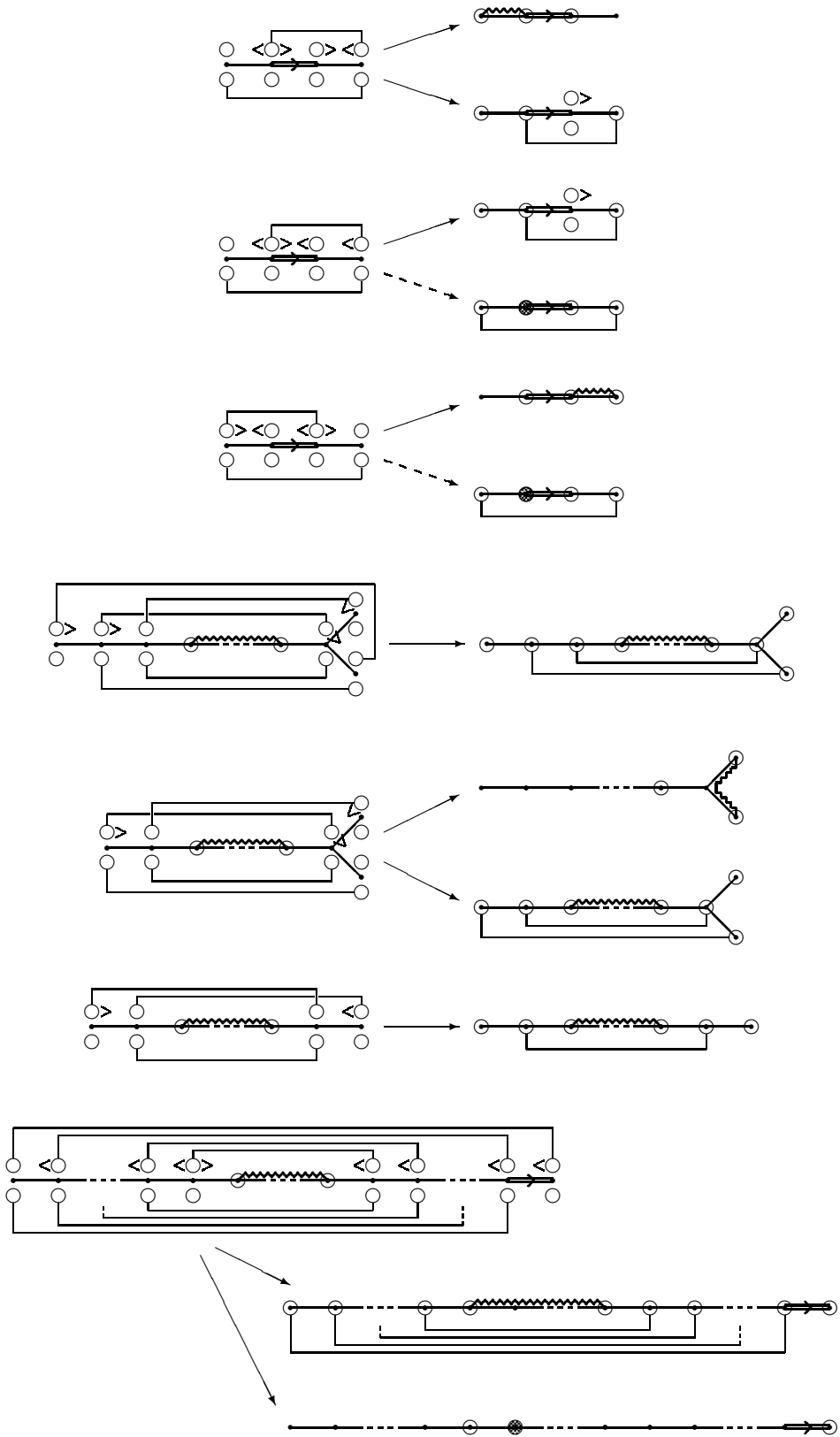
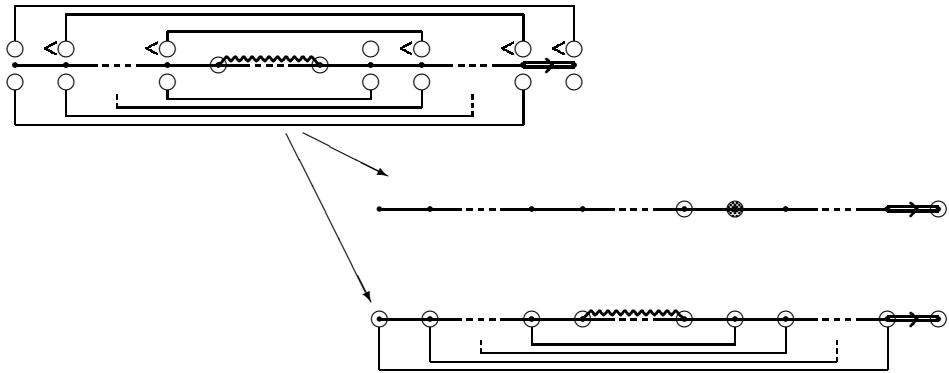
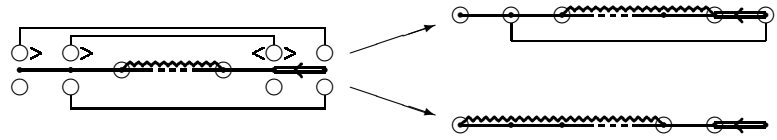
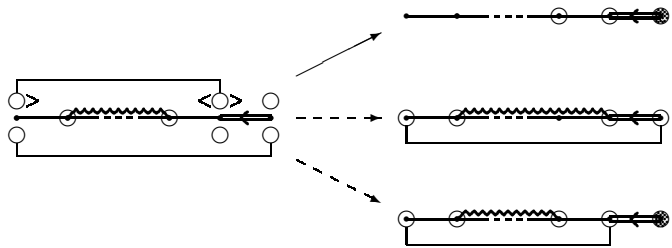
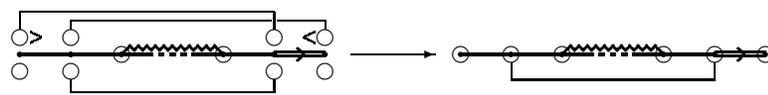
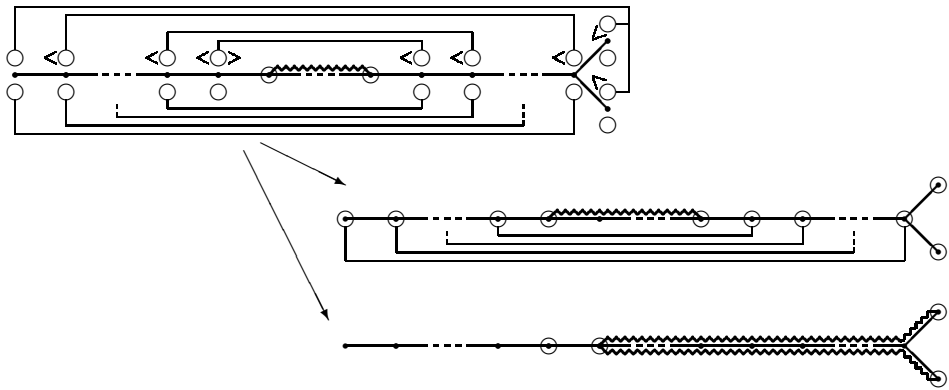
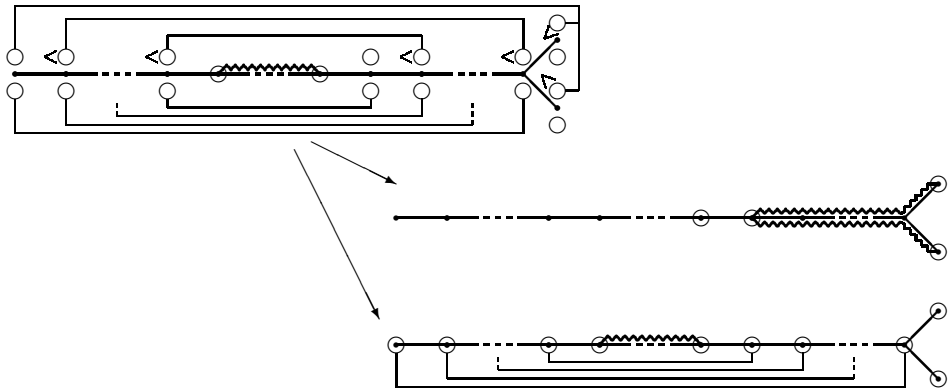
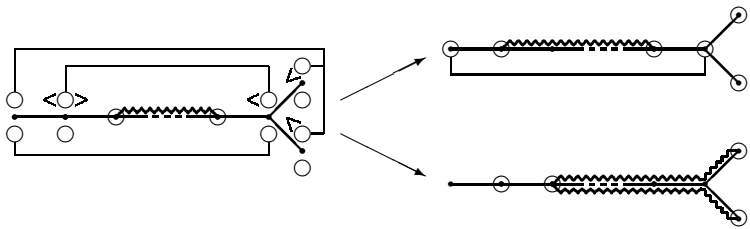
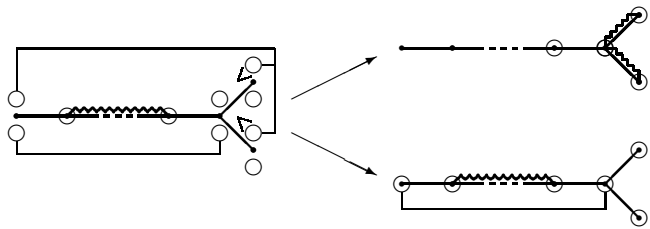
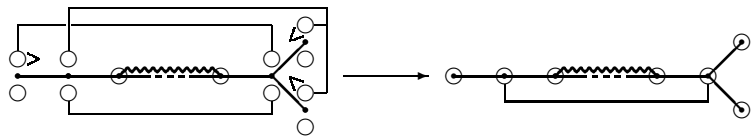
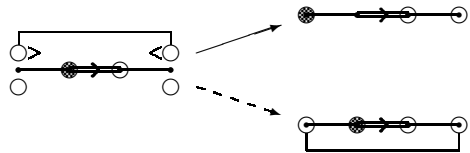
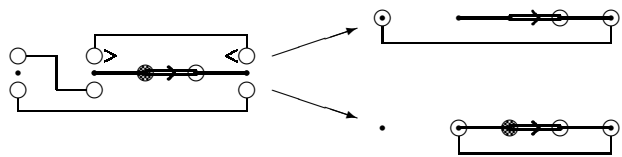
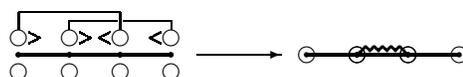
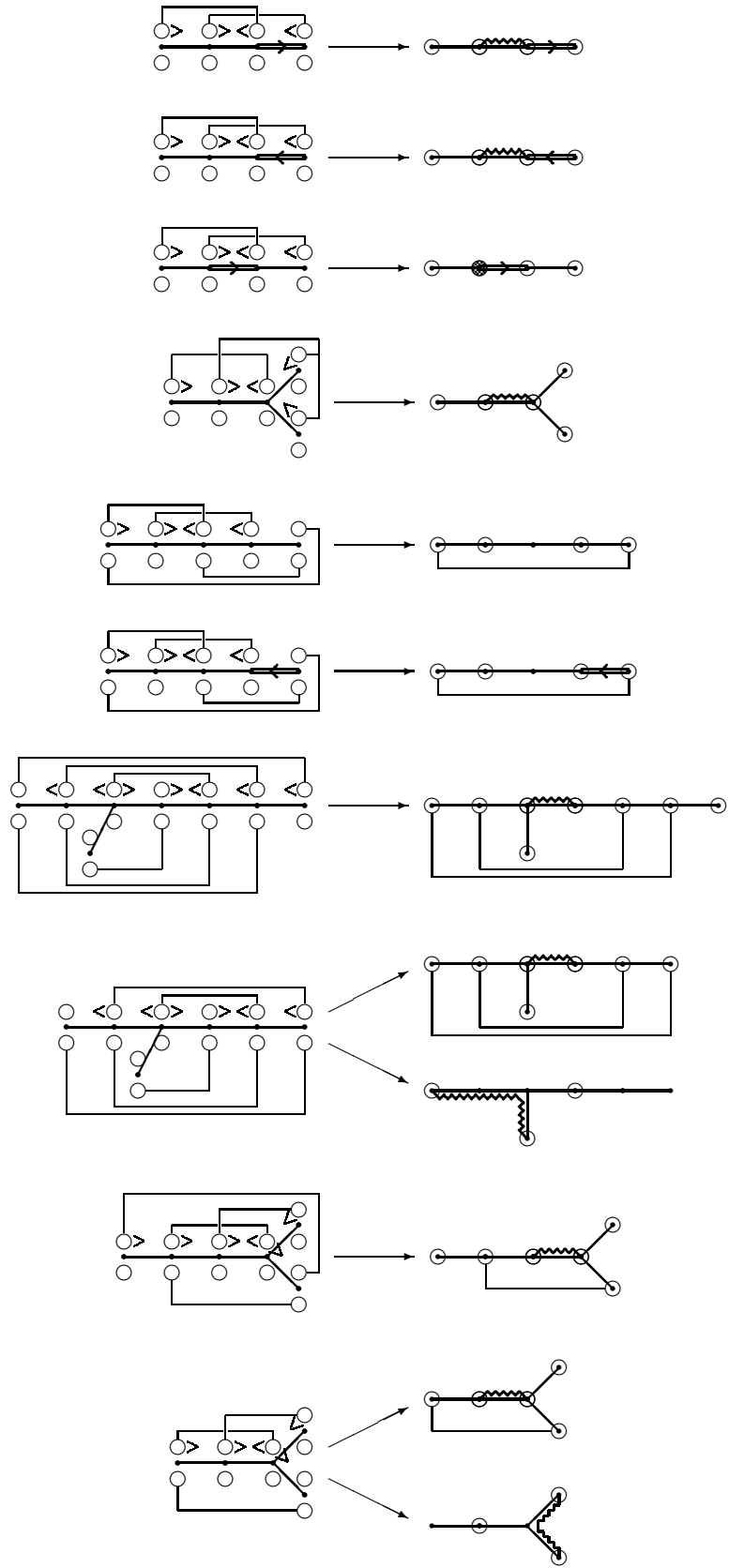
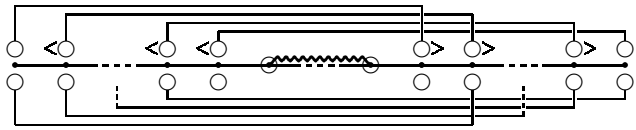
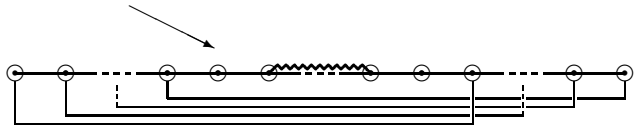
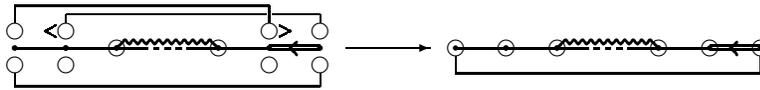
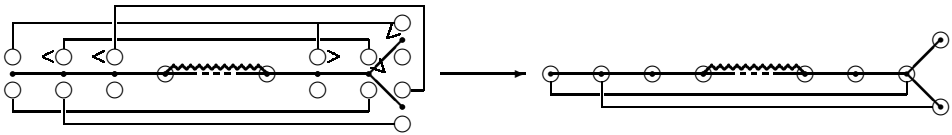
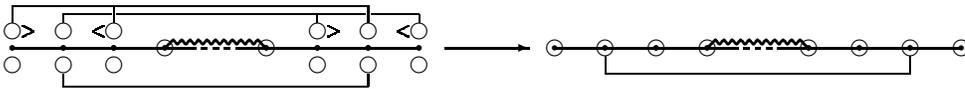
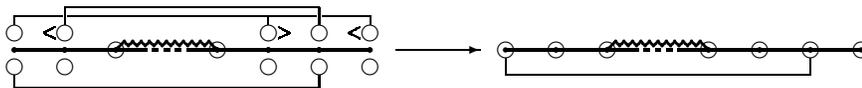
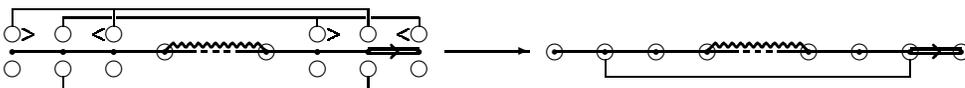
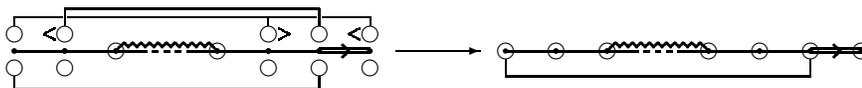
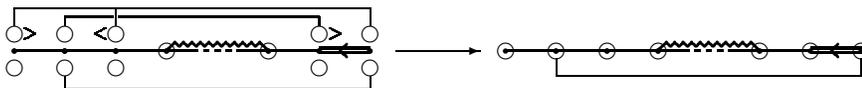
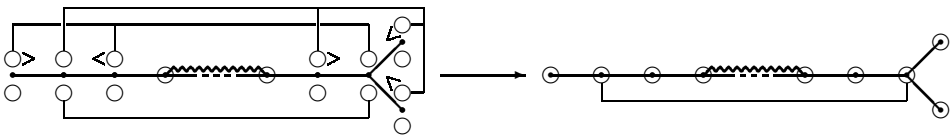
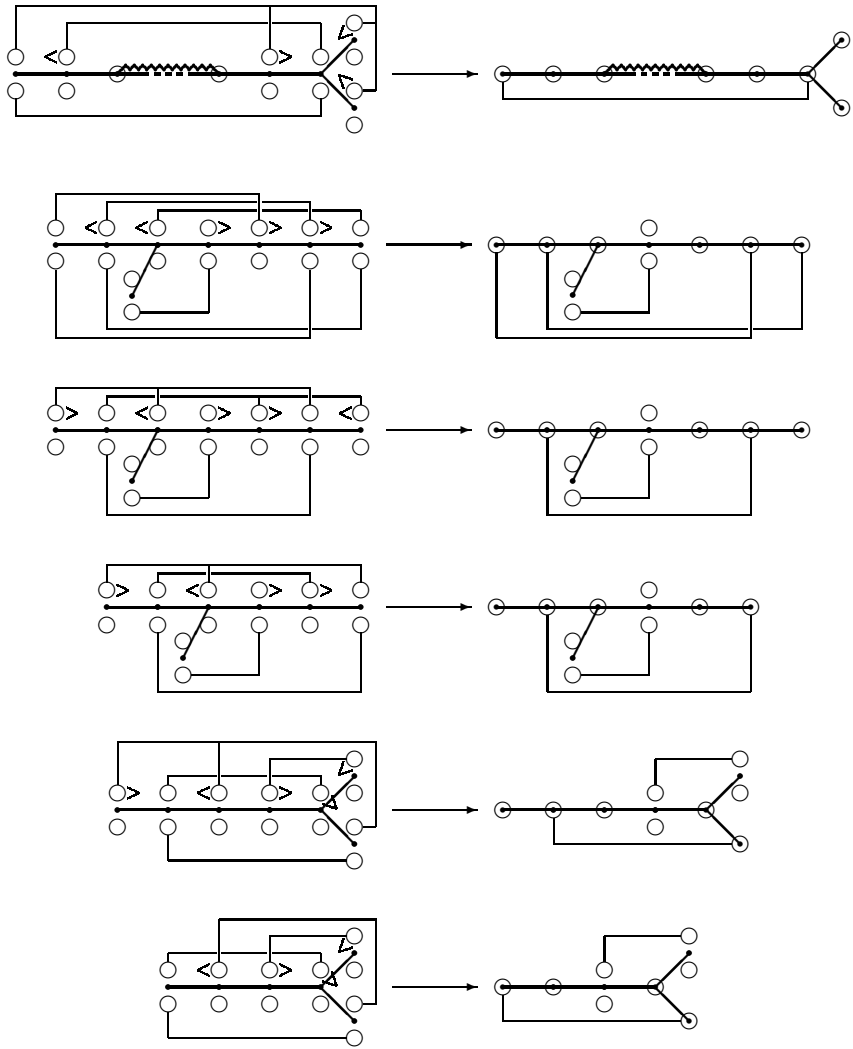
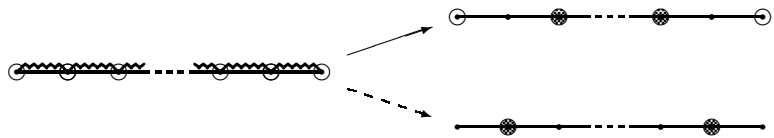
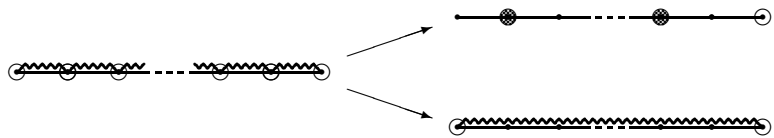
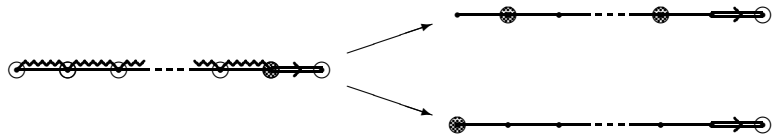
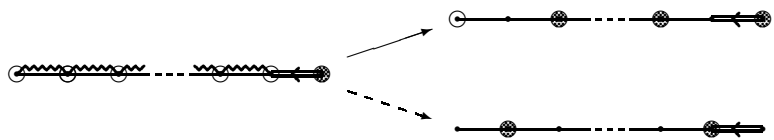
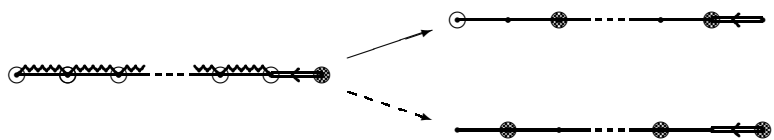
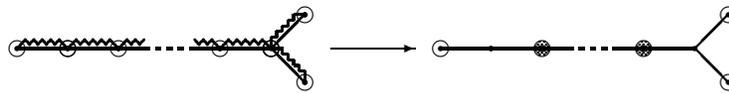
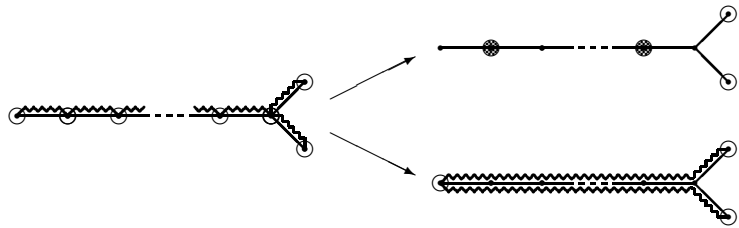
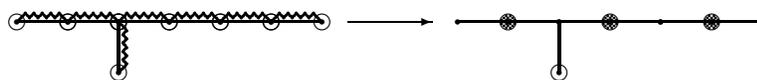
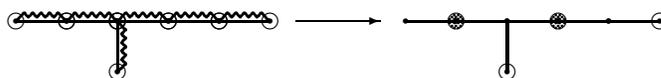
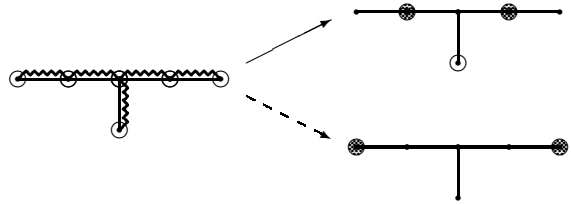
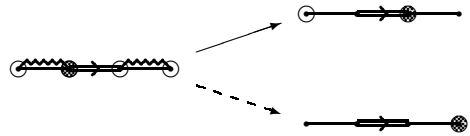
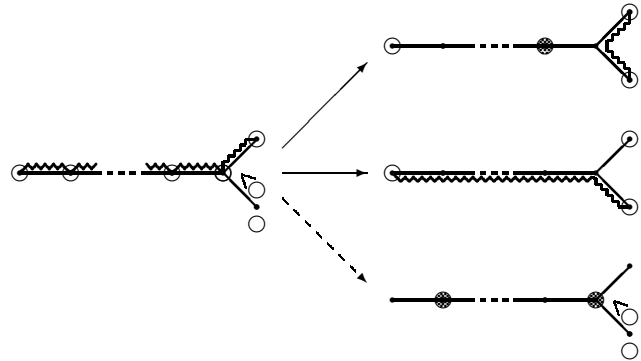
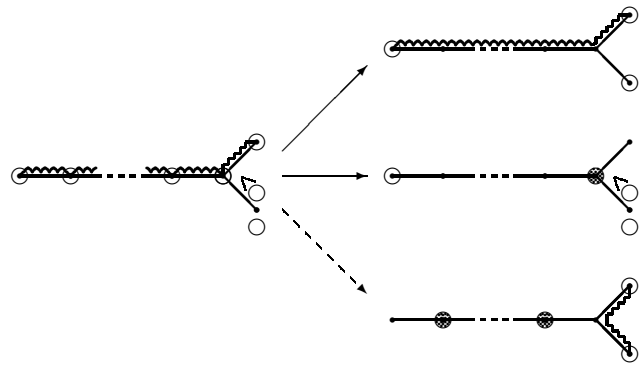
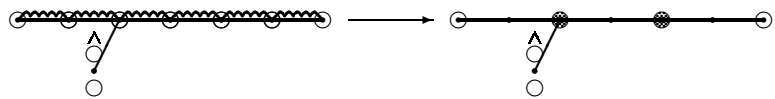
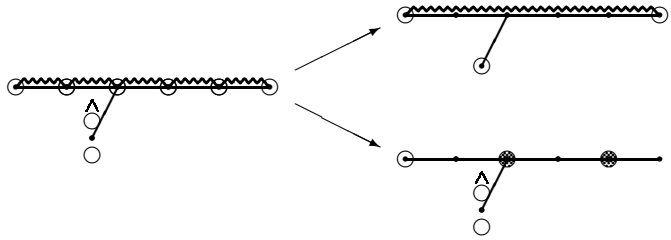
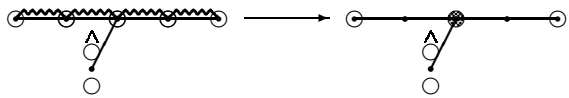
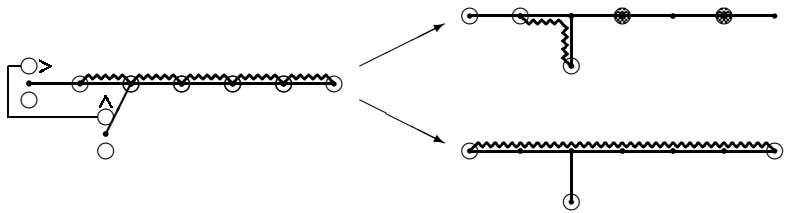
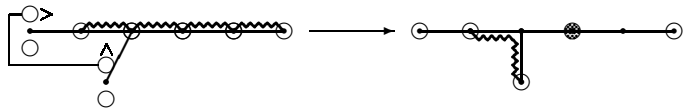
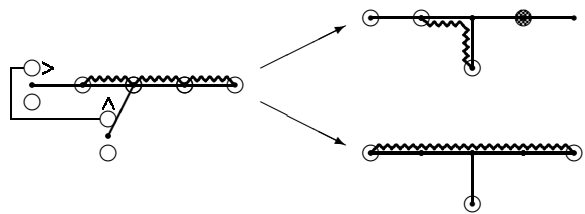
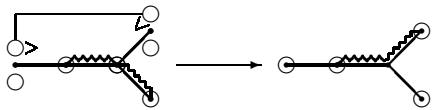
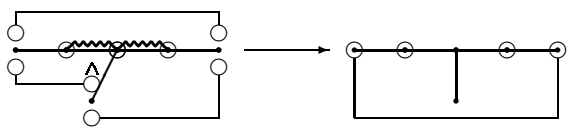
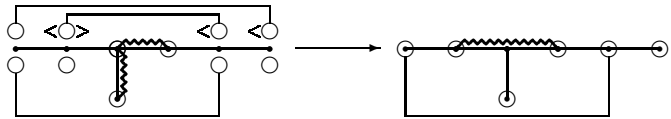
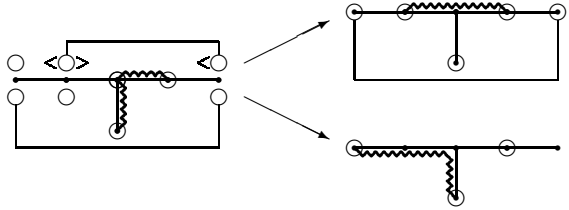
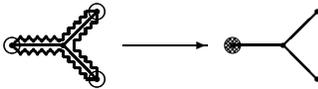
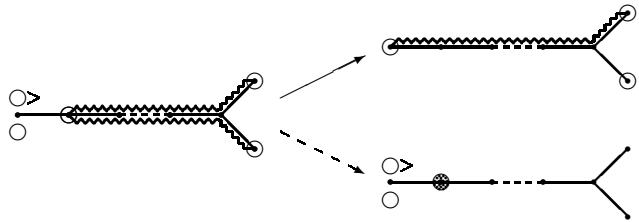
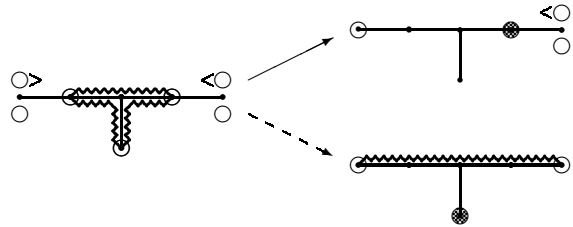
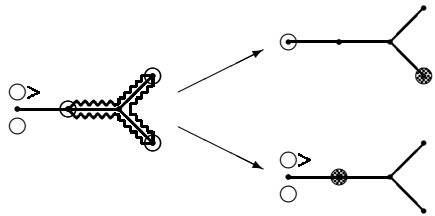
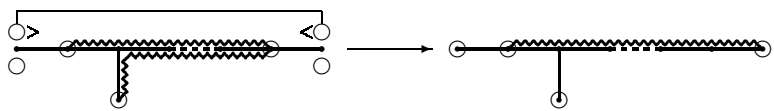
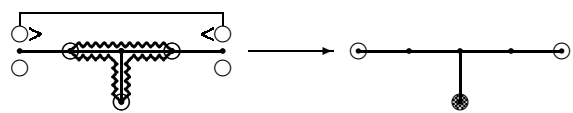
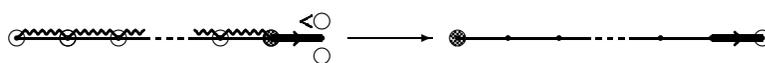
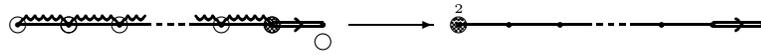
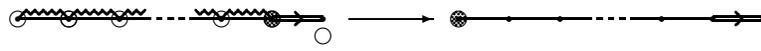
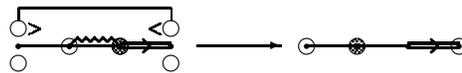
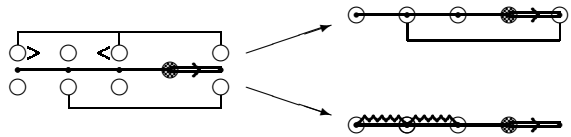
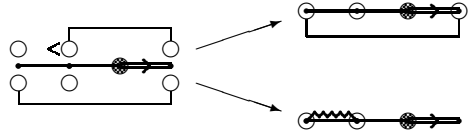
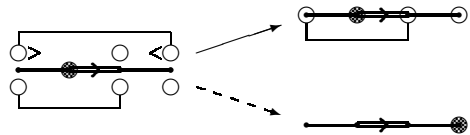
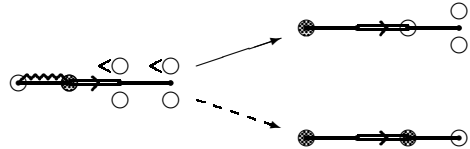
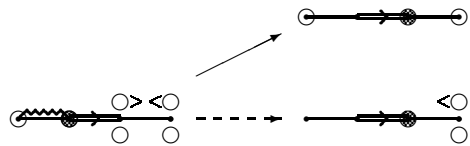
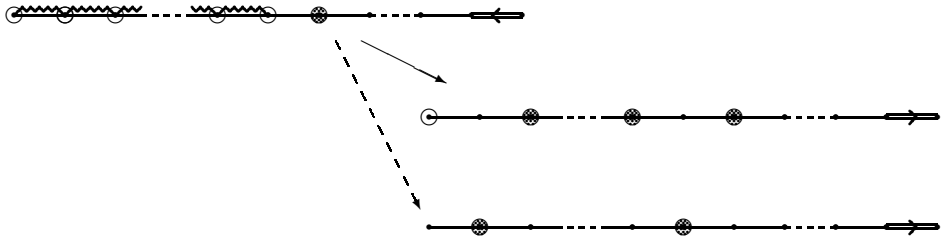
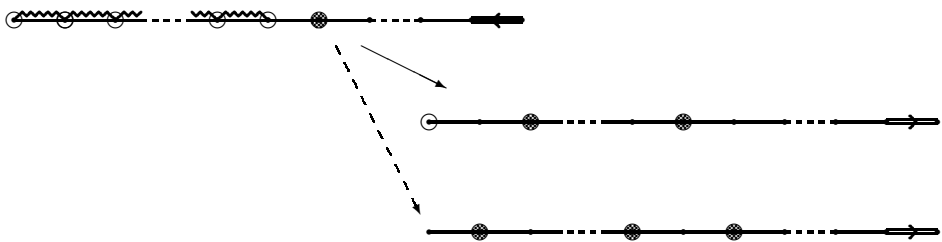
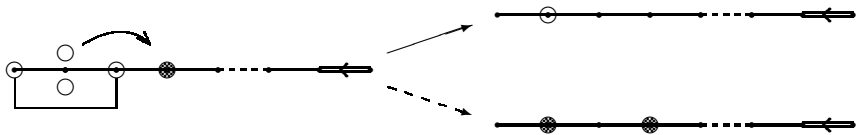
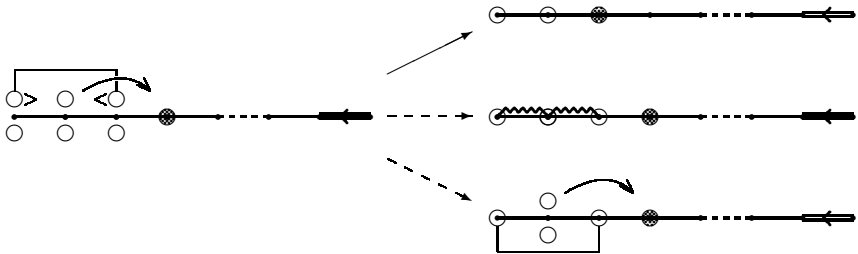
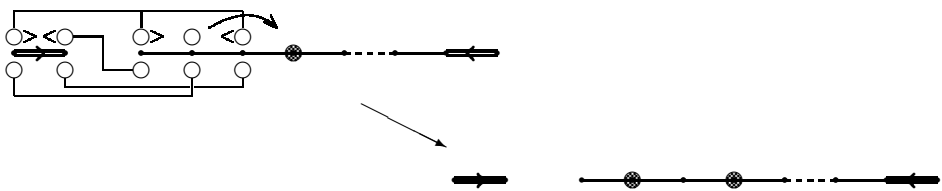
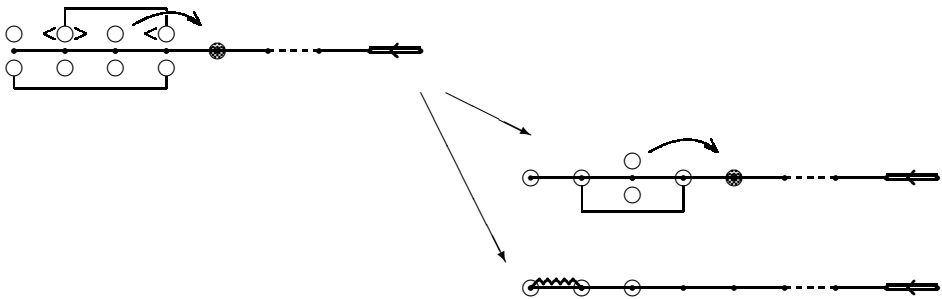
\[%
\]

\subsection*{Clan \fT}\

$\a^\fm(n)$, $n$ odd
\[%
\]

\subsection*{Primitive positive 1-combs}

For every positive comb $D$, $\{D\}$ is a minimal good distinguished subset of colors. We omit here the quotients given by positive combs.

\[%
\]

\end{document}